\pdfoutput=1
\documentclass[journal=jacsat,manuscript=article]{achemso}

\usepackage[version=3]{mhchem} % Formula subscripts using \ce{}
\usepackage{makecell}
\usepackage{graphicx}
\usepackage{mathtools}
\usepackage{amsmath}
\usepackage{amssymb}
\usepackage{subcaption}
\usepackage{natbib}

\newtheorem{rmk}{Remark}

\usepackage{siunitx}
\usepackage{float}
\usepackage{booktabs}
\usepackage[nomarkers]{endfloat} % or /usepackage{float}
\usepackage{color}
\usepackage[normalem]{ulem}	
\usepackage{soul}
\usepackage{color, colortbl}
\definecolor{Gray}{gray}{0.9}

\usepackage{paralist}
\usepackage{color, soul}
\usepackage[ruled,vlined]{algorithm2e}
\usepackage{multirow}
\usepackage[flushleft]{threeparttable}
\usepackage{textcomp}
\usepackage[numbers]{natbib}

\newtheorem{proposition}{Proposition}
\title{Fast Explicit Machine Learning-Based Model Predictive Control of Nonlinear Processes Using Input Convex Neural Networks}
\author{Wenlong Wang}
\affiliation[National University of Singapore]
{Department of Chemical and Biomolecular Engineering, National University of Singapore, 117585, Singapore}

\author{Haohao Zhang}
\affiliation[National University of Singapore]
{Department of Chemical and Biomolecular Engineering, National University of Singapore, 117585, Singapore}
\alsoaffiliation[CAS1]
{Key Laboratory of Green Process and Engineering, Institute of Process Engineering, Chinese Academy of Sciences, Beijing 100190, China}
\alsoaffiliation[CAS2]
{School of Chemical Engineering, University of Chinese Academy of Sciences, Beijing 100049, China}

\author{Yujia Wang}
\affiliation[National University of Singapore]
{Department of Chemical and Biomolecular Engineering, National University of Singapore, 117585, Singapore}

\author{Yuhe Tian}
\affiliation[West Virginia University]
{Department of Chemical and Biomedical Engineering, West Virginia University, Morgantown, WV 26506, United States}
\author{Zhe Wu}
\affiliation[National University of Singapore]
{Department of Chemical and Biomolecular Engineering, National University of Singapore, 117585, Singapore}
\email{wuzhe@nus.edu.sg}

\begin{document}

\begin{tocentry}
	\centering
	\includegraphics[width=\textwidth]{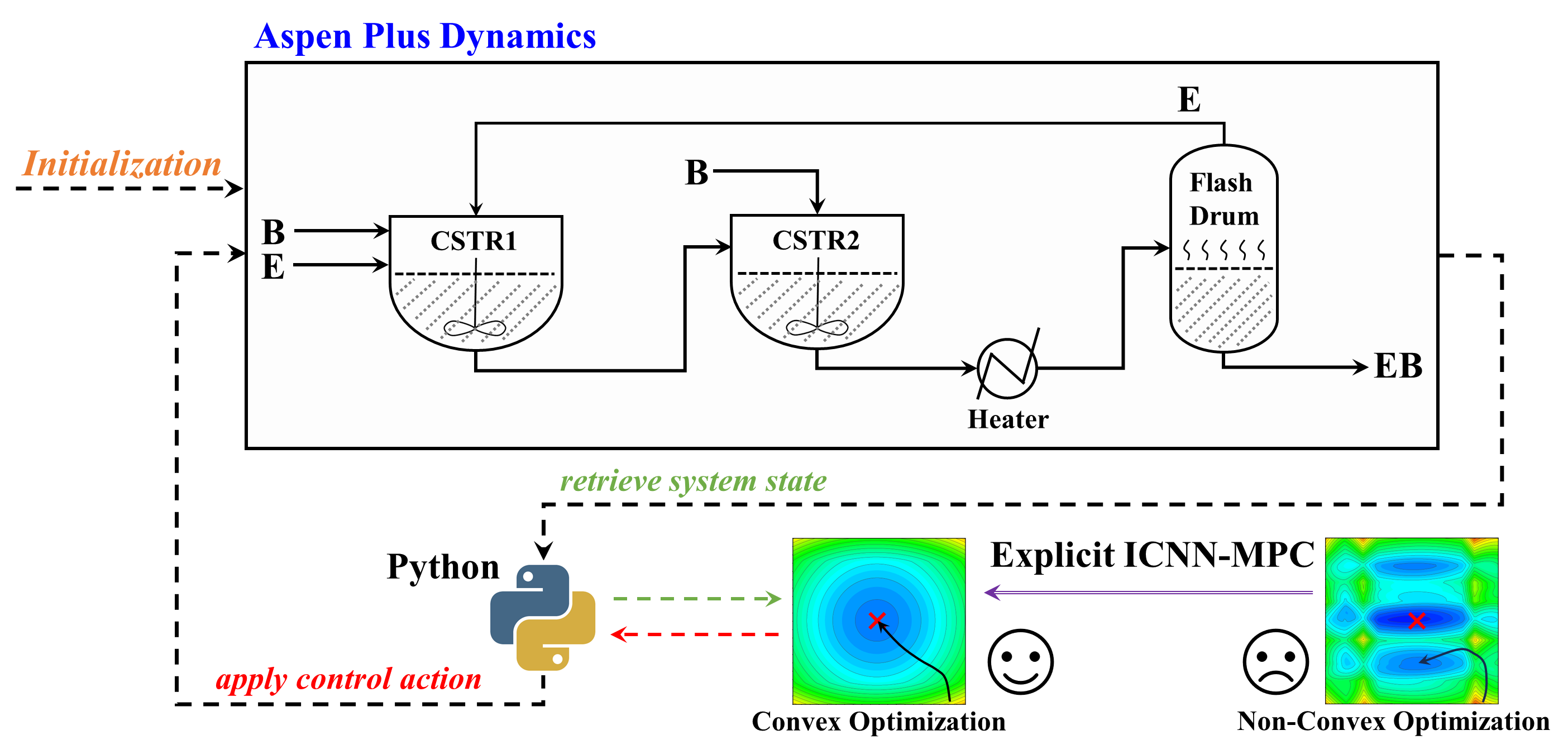}
\end{tocentry}

\begin{abstract}
Explicit machine learning-based model predictive control (explicit ML-MPC) has been developed to reduce the real-time computational demands of traditional ML-MPC. However, the evaluation of candidate control actions in explicit ML-MPC can be time-consuming due to the non-convex nature of machine learning models. To address this issue, we leverage Input Convex Neural Networks (ICNN) to develop explicit ICNN-MPC, which is formulated as a convex optimization problem. Specifically, ICNN is employed to capture nonlinear system dynamics and incorporated into MPC, with sufficient conditions provided to ensure the convexity of ICNN-based MPC. We then formulate mixed-integer quadratic programming (MIQP) problems based on the candidate control actions derived from the solutions of multi-parametric quadratic programming (mpQP) problems within the explicit ML-MPC framework. Optimal control actions are obtained by solving real-time convex MIQP problems. The effectiveness of the proposed method is demonstrated through two case studies, including a chemical reactor example, and a chemical process network simulated by Aspen Plus Dynamics, where explicit ML-MPC written in Python is integrated with Aspen dynamic simulation through a programmable interface. 
\end{abstract}


\section{Introduction} \label{sec:Intro}
Machine learning (ML) models have exhibited great potential in various industrial applications due to their ability to effectively model highly nonlinear processes in chemical plants and petroleum refineries where first-principles models are rarely available \cite{daoutidis2023future,pistikopoulos2024advanced}. Model predictive control (MPC), an advanced process control method, uses a predictive model for processes dynamics to obtain optimal control actions based on state measurements.\cite{morari1999MPC_review1,qin2003MPC_review2,mayne2014MPC_review3}. In recent years, machine learning-based model predictive control (ML-MPC) has been developed for nonlinear processes with data-driven models implemented by ML methods.\cite{wu2019ML-based_MPC_part1,wu2019ML-based_MPC_part2}. Although the use of ML models allows efficient modeling of nonlinear processes, it also poses challenges to the real-time implementation of ML-MPC due to the nonconvexity of ML models and the resulting non-convex ML-based optimization problems.\cite{wang2024explicitML-MPC}.\par

In our previous work, explicit ML-MPC\cite{wang2024explicitML-MPC} has been proposed to mitigate this issue following the idea of explicit MPC in which multi-parametric programming is utilized to convert real-time optimization problems into numerical evaluations\cite{pappas2021mpP_review-3}. A key step in the proposed explicit ML-MPC framework is the linearization of ML models via piecewise linear affine functions such that multi-parametric quadratic programming (mpQP) problems can be formulated for each segment of the discretized state-space. The solutions to mpQP problems provide a convenient and efficient way to obtain the optimal control actions based on real-time state measurements, compared to traditional implicit ML-MPC. However, the evaluation of candidate control actions in explicit ML-MPC can be time-consuming and computationally intensive due to the strong non-convexity introduced by ML models. Although the proposed explicit ML-MPC has managed to reduce the number of candidate control actions by converting the original continuous space to its discretized counterpart, finding the optimal one from these candidates can be quite challenging, as it involves solving a mixed-integer quadratic programming (MIQP) problem. The non-convexity of ML models results in a non-convex MIQP problem, where the global optimality of the solution cannot be guaranteed without employing advanced global optimizers\cite{burer2012MINLPs}.\par

Recent research on the development of convex ML models has made substantial progress. Specifically, Input Convex Neural Network (ICNN) has been proposed to ensure that the neural network output is a convex function with respect to its input\cite{amos2017inputconvex}. The convexity of ICNN is ensured by imposing a non-negative constraint on the learnable weight of each layer and requiring the activation functions that are used after each layer to be convex and non-decreasing. The convexity of ICNN comes at the cost of losing some representation ability. To address this issue, shortcuts that connect the input and hidden layers are implemented in ICNN to partially restore its representation ability. Using ICNN as the predictive model in MPC, ICNN-MPC can be developed as a convex optimization problem due to the convexity of ICNN\cite{chen2018ICNN-MPC_1,yang2021ICNN-MPC_2}. However, it should be noted that the convexity of ICNN does not necessarily lead to a convex objective function in ICNN-MPC, where a quadratic objective function is commonly adopted. Therefore, in this work, ICNN is further improved to ensure a convex objective function in ICNN-MPC.\par

To demonstrate the effectiveness of the proposed explicit ICNN-MPC in controlling chemical processes, Aspen Plus Dynamics, a high-fidelity process modeling and simulation platform in the chemical industry\cite{taqvi2017aspen_modelling}, is used to simulate a chemical process network under explicit ICNN-MPC. Although the joint simulation between Aspen Plus Dynamics and Matlab has been developed for optimization-based process control\cite{zhang2024matlab_aspen},  the co-simulation between Aspen Plus Dynamics and Python is still in its infancy. Specifically, due to the rapid development of ML techniques and ML-MPC\cite{wu2019ML-MPC1,wu2023ML-MPC2,xiao2023ML-MPC3},  Python has been widely used for modeling and control work due to its user-friendly features and many mature libraries, and can serve as a versatile programming language capable of coordinating information from various sources to perform more complex tasks.

 Motivated by the above considerations, in this work, we develop an explicit ICNN-MPC framework that first introduces ICNN into MPC and subsequently formulates convex MIQP problems to improve computational efficiency and obtain global optimal solutions. A dynamic simulation of chemical processes is conducted using Aspen, controlled by the explicit ICNN-MPC implemented in Python, with integration achieved through a programmable interface.
 Specifically, we show the sufficient conditions under which the real-time optimization problems in explicit ICNN-MPC (i.e., MIQP) are convex. A toy example is used to demonstrate the convexity of ICNN and the convexity of the objective function in ICNN-MPC. Then, the formulation of MIQP problems in explicit ICNN-MPC is shown, and different approaches to solving convex MIQP problems are discussed. Finally, two case studies, including a chemical reactor example and a chemical process network simulated by Aspen Plus Dynamics, are presented to demonstrate the effectiveness of the proposed method.

\section{Preliminaries}\label{sec:Preliminaries}
\subsection{ML Models for Nonlinear Dynamic Systems}\label{ML Modelling}
In this work, we focus on a general class of continuous-time nonlinear systems that are described by the following nonlinear first-order ordinary differential equations (ODEs):
\begin{equation}\label{eq:nonlin_sys}
\begin{aligned}
\dot{x}=F(x,u),\;x(t_0)=x_0
\end{aligned}
\end{equation}
\noindent
where $x\in\mathbb{R}^n$ denotes the system state and $u\in\mathbb{R}^m$ represents the manipulated input that is bounded by $\mathbb{U}\coloneqq \left\{ u^{min}_i\leq u_i\leq u^{max}_i,i=1,2,\ldots,m\right\}\subseteq\mathbb{R}^m$. The value of $F$ is assumed to be zero when $x$ and $u$ are both zero, which ensures that the origin $(0,0)$ is a steady state for the nonlinear system of Eq. \ref{eq:nonlin_sys}.\par

For complex and highly nonlinear processes, their first-principles models  represented by Eq. \ref{eq:nonlin_sys} may be difficult to derive, thus hindering the implementation of traditional MPC that uses first-principles models. Nevertheless, the sampled data collected from the dynamic operation of Eq. \ref{eq:nonlin_sys} can be utilized to develop ML models of the following form:
\begin{equation}\label{eq:fml}
\check{x}_{t+1},\check{x}_{t+2},\ldots,\check{x}_{t+k}=F_{ML}(x_t,u_t,u_{t+1},\ldots,u_{t+k-1}),\;k=1,2,3,\ldots
\end{equation}

\noindent where $x_t$ denotes the system state measured at time instant $t$, and $\check{x}_{t+1},\check{x}_{t+2},\ldots,\check{x}_{t+k}$ denote the system states predicted by ML models for time instants 
$t+1,t+2,\ldots,t+k$, respectively. The control actions $u_{t},u_{t+1},\ldots,u_{t+k-1}$ are intended to apply to the system at time instants $t,t+1,\ldots,t+k-1$, respectively. The time interval between two consecutive sampling steps is constant and is defined as the sampling period $\Delta_t$.

\subsection{Explicit Model Predictive Control (MPC)}\label{Explicit MPC}
Explicit MPC has been proposed to accelerate the real-time implementation of MPC by converting online optimization problem to numerical evaluation through multi-parametric programming\cite{dua2008mpP_review-1,pistikopoulos2012mpP_review-2,pappas2021mpP_review-3}. In conventional MPC (i.e., implicit MPC), the optimal control action for each time step is obtained by solving nonlinear optimization problem online, which could be time-consuming and computationally demanding when process dynamics is complex. To address this issue, explicit MPC was developed based on multi-parametric programming that reveals the relationship between the optimal solutions and the variation of parameters in optimization problems\cite{oberdieck2016mpP,tian2021simultaneous}. A general form of nonlinear multi-programming problem is given as follows:
\begin{equation}\label{eq:mpP}
\begin{aligned}
\mathop{\min}_{u}\;&f_{obj}(x,u)\\
\text{s.t. }  &g_{ineq,i}(x,u)\leq 0,\;\forall\;i=1,2,3,\dots\\
              &h_{eq,j}(x,u)=0,\;\forall\;j=1,2,3,\dots\\
              &x\in\mathbb{X}\subseteq\mathbb{R}^n\\
              &u\in\mathbb{U}\subseteq\mathbb{R}^m
\end{aligned}
\end{equation}
\noindent
where $u$ represents the decision variables (e.g., control action), $x$ denotes the uncertain parameters (e.g., system state), and $f_{obj}$ is the nonlinear parametric objective function. $g_{ineq,i}$ and $h_{eq,j}$ are the $i$th inequality constraints and $j$th equality constraints, respectively. Note that $f_{obj}$, $g_{ineq}$, and $h_{eq}$ are required to be twice continuously differentiable in $x$ and $u$ to ensure the validity of Basic Sensitivity Theorem in multi-parametric programming\cite{fiacco1976sensitivity}. The solutions to Eq. \ref{eq:mpP} are in the form of $u^\ast=\Theta_k(x),\;\forall\;x\in\mathbb{X}_k,\;k=1,2,\ldots,N$, where $\mathbb{X}_1\bigcup\mathbb{X}_2\bigcup\dots\bigcup\mathbb{X}_N=\mathbb{X}$, and $N$ is the number of critical regions within the operating region $\mathbb{X}$ we considered. $\Theta_k$ is the $k$th mapping function that explicitly correlates the uncertain parameters $x$ that fall in $\mathbb{X}_k$ to the corresponding optimal solutions $u^{\ast}$. Under explicit MPC, solving real-time nonlinear optimization problems is no longer required for MPC implementation due to the use of these precomputed mapping functions, $\Theta(x)$\cite{pistikopoulos2020book}.

\section{Explicit ICNN-MPC for Nonlinear Processes}
In this section, we develop the explicit ICNN-MPC framework for nonlinear processes. Specifically, the architecture of ICNN is first introduced, followed by a nonlinear numerical example to show the convexity of the resulting optimization problem of ML-MPC due to the use of ICNN. Then, the development of explicit ICNN-MPC is presented, and we show that the evaluation of candidate control actions in explicit ICNN-MPC is reformulated to a convex MIQP problem that can be solved efficiently by various convex optimization algorithms. In this work, the variables in boldface represent matrices, and the non-bold variables represent vectors and scalars.

\subsection{Input Convex Neural Network (ICNN)}\label{ICNN}
 Input Convex Neural Network was originally proposed by Amos et al. to construct a neural network whose output is a convex function of the input\cite{amos2017inputconvex}. The model architecture of ICNN, as shown in Fig. \ref{fig:1}, is elaborately designed to exhibit convexity without significantly compromising its representational capacity. The forward propagation in ICNN is given by the following equation:
\begin{equation}\label{eq:icnn}
z_{i+1}=\sigma_i\left(\boldsymbol{W_{z,i}}z_i+\boldsymbol{W_{s,i}}s+b_i\right),\;i=0,1,\ldots,k-1
\end{equation}
where $\boldsymbol{W_{z,i}}$ and $\boldsymbol{W_{s,i}}$ are the $i$th weight terms associated with hidden state $z_i$ and input $s$, respectively. $b_i$ refers to the $i$th bias term, and $\sigma_i$ denotes the activation function of the $i$th hidden layer. To ensure that the output $z_k$ is a convex function of the input $s$, the weight terms $\boldsymbol{W_{z,0}},\boldsymbol{W_{z,1}},\ldots,\boldsymbol{W_{z,k-1}}$ are designed to be non-negative. In this work, this is achieved by changing the negative parameters in these weight terms to zero after each optimizer step, using the Pytorch framework. For TensorFlow users, this non-negative constraint on the trainable weights can be enforced by explicitly setting keyword argument ``constraint=tf.keras.constraints.NonNeg()'' during the training process\cite{wang2023ICLSTM}. Additionally, the activation functions $\sigma_{0},\sigma_{1},\ldots,\sigma_{k-1}$ are required to be convex and non-decreasing. All the elements in $z_0$ and $\boldsymbol{W_{z,0}}$ are initialized as zero. There are no constraints on the weights of the shortcuts (i.e., the dash lines in Fig. \ref{fig:1}) that connect input and hidden layers since adding linear terms to a convex function does not violate its convexity. The constraints on activation functions can be satisfied by some commonly used activation functions, for example, rectified unit (ReLU), leaky rectified unit, softplus, and max-pooling unit.\par
\begin{figure}[h]
\includegraphics[width=\textwidth]{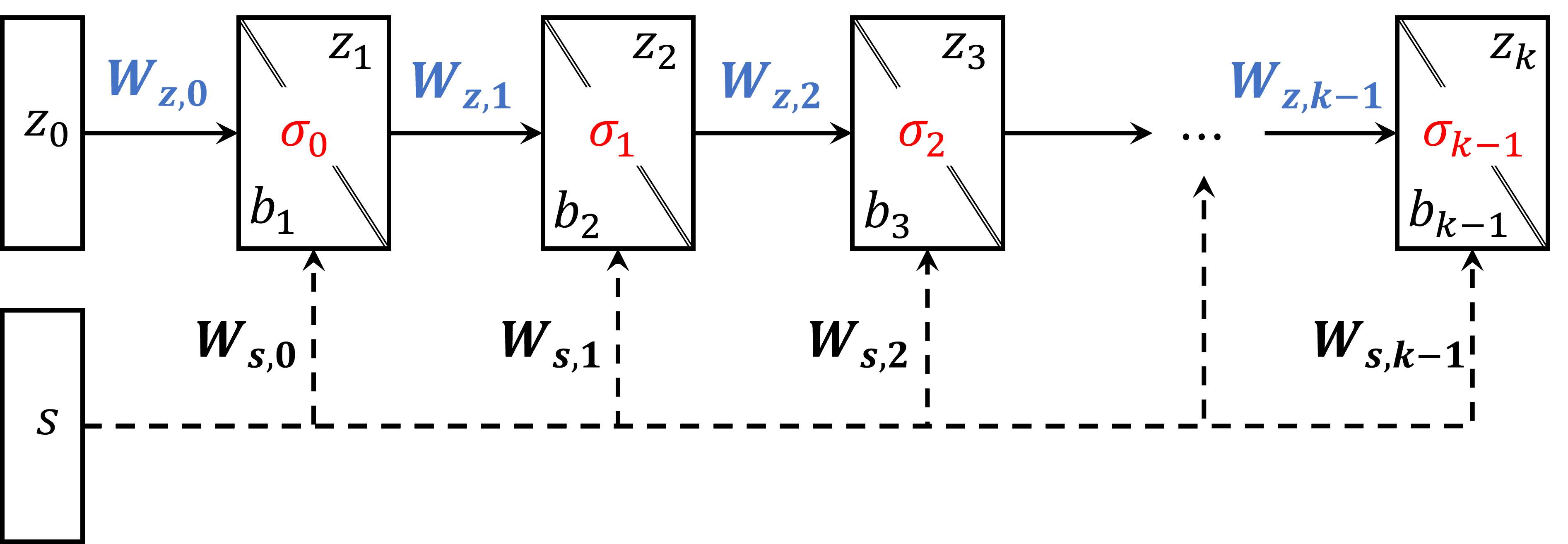}
\centering
\caption{Model architecture of ICNN.}
\label{fig:1}
\end{figure}

\subsection{ICNN-MPC}\label{ICNN-MPC}
 ICNN can be readily incorporated into MPC following the general formulation of ML-MPC where ICNN is used as the prediction model to capture the dynamics of nonlinear systems. However, the convexity of ICNN does not necessarily lead to a convex objective function in ICNN-MPC for any quadratic objective function. We first show the sufficient conditions for obtaining a convex optimization problem of ICNN-MPC. The design of ICNN-MPC is given by the following optimization problem:
\begin{equation}\label{eq:ICNN-MPC}
\begin{aligned}
\mathcal{J}=\mathop{\min}_{u_{t|t},\ldots,u_{t+N_p-1|t}}\;&\sum\limits_{k=1}^{N_p}\bar{x}^\top_{t+k|t}\boldsymbol{M}\bar{x}_{t+k|t}+\sum\limits_{k=0}^{N_p-1}u^\top_{t+k|t}\boldsymbol{N}u_{t+k|t}\\
\text{s.t. }~~
        &\bar{x}_{t+k|t}=F_{ICNN,k}(x_{t|t},u_{t|t},\ldots,u_{t+k-1|t}),\;k=1,2,\ldots,N_p\\
        &x_{t|t}=x(t)\\
        &x\in\mathbb{X}\subseteq\mathbb{R}^n\\
        &u\in\mathbb{U}\subseteq\mathbb{R}^m
\end{aligned}
\end{equation}
where $t+k|t$ denotes the prediction made for the instant $t+k$ based on the measurement at the instant $t$, and $N_p$ represents the length of prediction horizon. $\bar{x}_{t+k|t}$ denotes the absolute value of the predicted system state for time instant $t+k$ and $F_{ICNN,k}$ is the $k$th  ICNN model that predicts $\bar{x}_{t+k|t}$ by using $x_{t|t},u_{t|t},\ldots,u_{t+k-1|t}$. Note that to implement ICNN-MPC and ensure convexity, we need to train $N_p$ ICNN models that predict the state at $t=t+1$, $t=t+2$,..,$t=t+N_p$, respectively; however, they can be trained concurrently using parallel computing to save time. 
$\boldsymbol{M}$ and $\boldsymbol{N}$ are weight matrices for the state and control action, respectively. $x_{t|t}$ is the state measured at time instant $t$. $\mathbb{X}$ and $\mathbb{U}$ are assumed to be convex sets.  The following proposition establishes the sufficient conditions that ensure the convexity of the optimization problem of ICNN-MPC.
\begin{proposition}
    Consider the ICNN-MPC of Eq. \ref{eq:ICNN-MPC} using the ICNN of Eq. \ref{eq:icnn}  to predict the system dynamics.   Eq. \ref{eq:ICNN-MPC} is a convex optimization problem if the following three conditions are satisfied: (1) the output of ICNN is non-negative; (2) $\boldsymbol{M}$ is a diagonal matrix whose diagonal elements are all positive; (3) all the constraints are convex.
\end{proposition}

\begin{proof}
The third condition is trivial given that $\mathbb{X}$ and $\mathbb{U}$ in Eq. \ref{eq:ICNN-MPC} are designed to be convex sets. Therefore, we present the proof for the first and second conditions. We first show that the square of a non-negative convex function still reserves convexity. Let $f : \mathbb{R}^{n+m\times N_p} \to \mathbb{R}^n$ be a convex function learned by an ICNN model with $x=f(s)\geq0$ where $x$ and $s$ are vectors of appropriate dimensions. Let $g : \mathbb{R}^{n} \to \mathbb{R}$ be a quadratic function in ICNN-MPC such that $g(x)= x^\top\boldsymbol{M}x$ where $\boldsymbol{M}$ is a diagonal matrix with all its diagonal elements $M_{1},M_{2},\ldots,M_{n}$ positive. Then, $g(s)=f(s)^\top\boldsymbol{M}f(s)$ can be reformulated as the sum of $n$ terms: $g(s)=g_1(s)+g_2(s)+\ldots+g_n(s)=M_1[f_1(s)]^2+M_2[f_2(s)]^2+\ldots+M_n[f_n(s)]^2$ where $f_1,f_2,\ldots,f_n$ are independent components of $f$. Each component of $f$ is a function of $\mathbb{R}^{n+m\times N_p} \to \mathbb{R}$. The Hessian matrix for the $i$th term of $g(s)$ is $\boldsymbol{H}(g_i(s))=2M_i\left[\nabla_s f_i(s)\nabla_s f_i(s)^\top+f_i(s)\boldsymbol{H}(f_i(s))\right]$, where $\nabla_s f_i(s)$ denotes the gradient of the $i$th component of $f$ with respect to $s$. $\boldsymbol{H}(g_i(s))$ is positive semi-definite due to the fact that $M_i$ is positive, $f_i(s)$ is non-negative, and $\nabla_s f_i(s)\nabla_s f_i(s)^\top$ and $\boldsymbol{H}(f_i(s))$   are both positive semi-definite, which makes $g_i(s)$ a convex function with respect to $s$. Subsequently, $g(s)=g_1(s)+g_2(s)+\ldots+g_n(s)$ exhibits convexity because the summation of convex functions is still a convex function.
%Let $f : \mathbb{R}^{n+m*N_p} \to \mathbb{R}^n$ be a convex function learned by an ICNN model with $f(x)\geq0$, and $g : \mathbb{R}^{n} \to \mathbb{R}$ be a quadratic function in ICNN-MPC (i.e., $g(x)=[f(x)]^2$). The second derivative of $g$ is $g^{''}=2[f^{'}]^2+2f\times f^{''}$, which is always non-negative due to the fact that $[f^{'}(x)]^2\geq0$, $f(x)\geq0$, and $f^{''}(x)\geq0$. Therefore, $g(x)=[f(x)]^2$ is a convex function with respect to $x$. Subsequently, the objective function $\mathcal{J}$ exhibits convexity because the summation of convex functions is still a convex function.\par

The non-negative constraint on the output of ICNN restricts the fitting ability of ICNN as system state may take negative values in state-space. To address this issue, we improve the design of ICNN to predict the absolute value of the state. Specifically, sampled data obtained from computer simulations or experiments is first mapped to their absolute counterpart. Activation functions that eliminate negative numbers (e.g., ReLU) are chosen for $\sigma_{k-1}$. This strategy will invalidate the recursive use of ICNN when the prediction horizon $N_p$ is larger than one, as ICNN requires the exact value of state (i.e., not its absolute counterpart) as input. This problem can be solved by training $N_p$ ML models for each predicted system states at $t+1, t+2,\ldots, t+N_p$, respectively, using the the ML model shown in Section ``ML Models for Nonlinear Dynamic Systems''. Specifically, when $\boldsymbol{M}$ is a diagonal matrix with all its diagonal elements $M_{1},M_{2},\ldots,M_{n}$ positive as mentioned above, then equation $\bar{x}^\top_{t+k|t}\boldsymbol{M}\bar{x}_{t+k|t}=x^\top_{t+k|t}\boldsymbol{M}x_{t+k|t}$ holds for all values of $k$ shown in Eq. \ref{eq:ICNN-MPC}. The is true because $\bar{x}^\top_{t+k|t}\boldsymbol{M}\bar{x}_{t+k|t}=\sum\limits_{i=1}^{n}M_{i}\bar{x}^2_{i,t+k|t}=\sum\limits_{i=1}^{n}M_{i}x^2_{i,t+k|t}=x^\top_{t+k|t}\boldsymbol{M}x_{t+k|t}$, where $\bar{x}_{i,t+k|t}$ and $x_{i,t+k|t}$ are the $i$th component of $\bar{x}_{t+k|t}$ and $x_{t+k|t}$, respectively. The matrix $\boldsymbol{M}$ designed in this manner ensures that the values of the objective function remain unaffected despite the non-negative output constraint imposed on the ICNN. Additionally, it ensures that the value of objective function equals zero only at the origin, and is greater than zero for all other states (i.e., $x^\top\boldsymbol{M}x>0,\;\forall\;x\in\mathbb{X}\;\backslash\;\{\mathbf{0}\}$). Therefore, if the three conditions in the theorem statement are met (i.e., an ICNN model is designed with non-negative outputs, the weight matrix $\boldsymbol{M}$ for system state in the ICNN-MPC is a diagonal matrix with all its elements positive, and convex constraint sets), then the ICNN-MPC of Eq. \ref{eq:ICNN-MPC} is a convex optimization problem.\hfill$\blacksquare$ 
\end{proof}

To better demonstrate the effectiveness of using the ICNN-MPC of Eq. \ref{eq:ICNN-MPC} to obtain a convex approximation of non-convex functions, the design of ICNN-MPC for a numerical example is presented. The system dynamics of this toy example is described by the following nonlinear equations:
\begin{equation}\label{eq:small-demo}
\begin{aligned}
&x_{1,t+1}=0.5x^2_{1,t}-x_{2,t}+\sin{u_{1,t}}-\cos{u_{2,t}}\\
&x_{2,t+1}=-x_{1,t}+0.5x^2_{2,t}-\cos{u_{1,t}}+\sin{u_{2,t}}
\end{aligned}
\end{equation}
where $x_{t}=[x_{1,t},x_{2,t}]^\top$ is  the  state vector and $u_{t}=[u_{1,t},u_{2,t}]^\top$ is   the manipulated input vector. An ICNN model $F^\ast_{ICNN}$ that uses $x_{t|t}$ and $u_{t|t}$ to predict $\bar{x}_{t+1|t}$ is developed to capture the dynamics of the nonlinear system. The length of the prediction horizon is set to 1 in order to visualize the objective function. The   corresponding ICNN-MPC is given by the following convex optimization problem:
\begin{equation}\label{eq:ICNN-MPC-demo}
\begin{aligned}
\mathcal{J}=\mathop{\min}_{u_{t|t}}\;&\bar{x}^\top_{t+1|t}\boldsymbol{M}\bar{x}_{t+1|t}+u^\top_{t|t}\boldsymbol{N}u_{t|t}\\
\text{s.t. }~~
        &\bar{x}_{t+1|t}=F^\ast_{ICNN}(x_{t|t},u_{t|t})\\
        &x_{t|t}=x(t)\\
        &-10\leq u_{t|t}\leq 10
\end{aligned}
\end{equation}
where $\boldsymbol{M}=[1\;0;0\;1]$ and $\boldsymbol{N}=[0.1\;0;0\;0.1]$. We first conduct extensive open-loop simulations based on Eq. \ref{eq:small-demo} to generate the training data for the ICNN model $F^\ast_{ICNN}$ to capture the nonlinear dynamics. Note that the sample-label pair is in the form of \{$(x_{t|t},u_{t|t})$,\;$\bar{x}_{t+1|t}$\} where $\bar{x}_{t+1|t}$ is the absolute value of $x_{t+1|t}$. Then, an ICNN model is trained to predict $\bar{x}_{t+1|t}$ by utilizing $x_{t|t}$ and $u_{t|t}$ as inputs. For comparison purposes, another ICNN model and a feedforward neural network (FNN) model, both trained on \{$(x_{t|t},u_{t|t})$,\;$x_{t+1|t}$\} (i.e., using the true value of $x_{t+1|t}$ as outputs), are also developed. To demonstrate the convexity of the ICNN model, the prediction profiles given by the ICNN model with absolute output values and the FNN model for the first component of $x_{t+1|t}$ with respect to $u_{t|t}$ for $x_{t|t}=[1.5,-1]$ are shown in Fig. \ref{fig:2s}. The profiles of the objective function with respect to $u_{t|t}$ for $x_{t|t}=[1.5,-1]$ are shown in Fig. \ref{fig:2}.\par
\begin{figure}[h]
\centering
\begin{subfigure}{0.495\textwidth}
    \includegraphics[width=\textwidth]{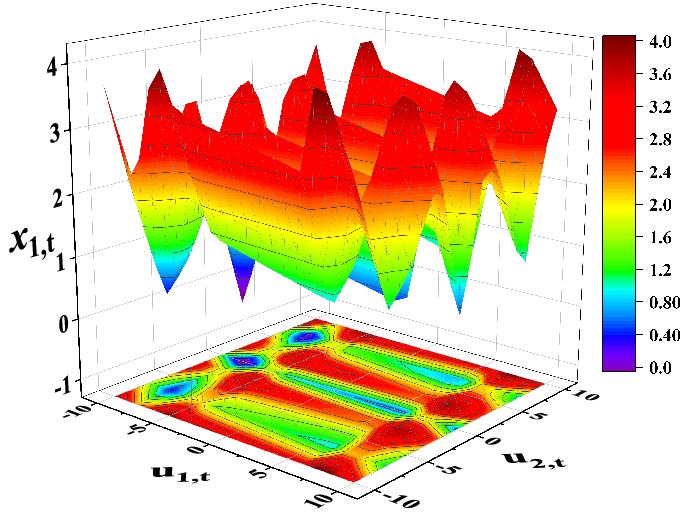}
    \caption{System dynamics profiles described by an FNN model.}
    \label{fig:2s_a}
\end{subfigure}
\hfill
\begin{subfigure}{0.495\textwidth}
    \includegraphics[width=\textwidth]{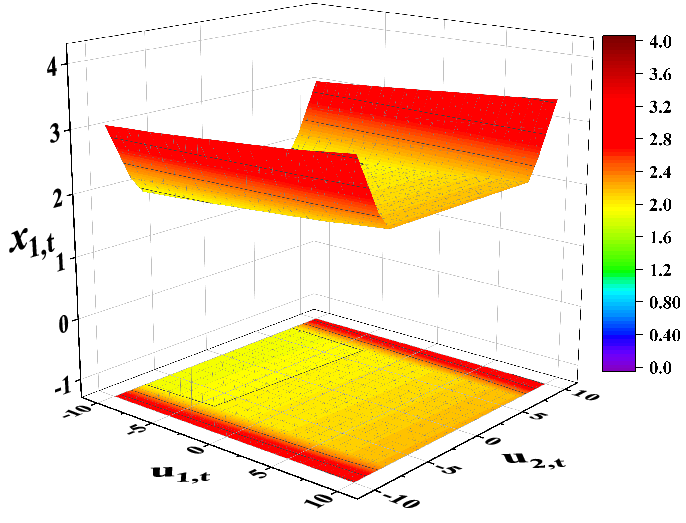}
    \caption{System dynamics profiles described an ICNN model.}
    \label{fig:2s_b}
\end{subfigure}
\hfill
\caption{System dynamics profiles described by an FNN model and an ICNN model.}
\label{fig:2s}
\end{figure}

\begin{figure}[h]
\centering
\begin{subfigure}{0.495\textwidth}
    \includegraphics[width=\textwidth]{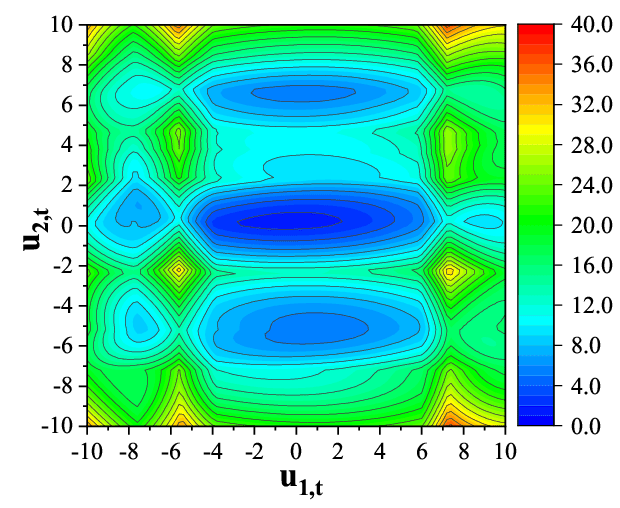}
    \caption{Profiles of the objective function obtained by the FNN model that is trained without non-negative output constraint.}
    \label{fig:2_a}
\end{subfigure}
\hfill
\begin{subfigure}{0.495\textwidth}
    \includegraphics[width=\textwidth]{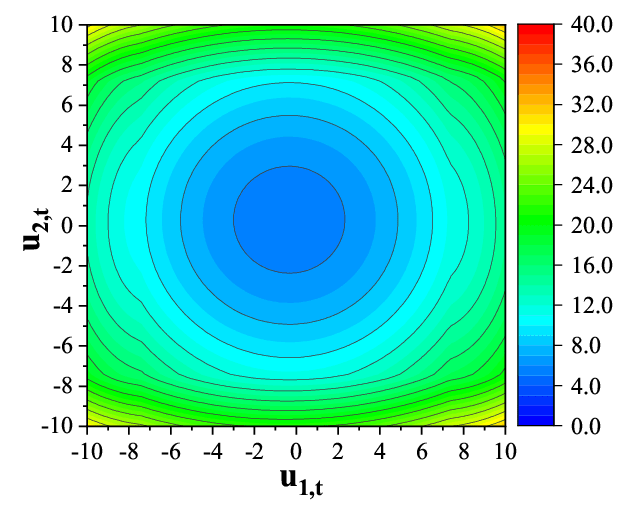}
    \caption{Profiles of the objective function obtained by the ICNN model that is trained without non-negative output constraint.}
    \label{fig:2_b}
\end{subfigure}
\hfill
\newline
\begin{subfigure}{0.495\textwidth}
    \includegraphics[width=\textwidth]{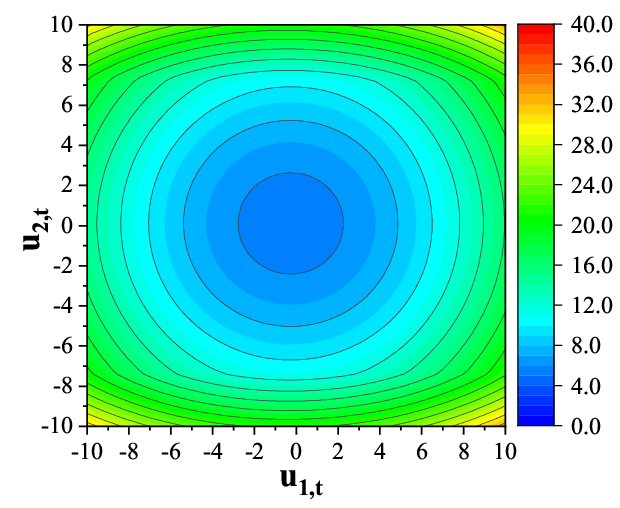}
    \caption{Profiles of the objective function obtained by the ICNN model that is trained with non-negative output constraint.}
    \label{fig:2_c}
\end{subfigure}
\caption{Profiles of the objective functions under three different training schemes.}
\label{fig:2}
\end{figure}
It is shown in Fig. \ref{fig:2s} that the ICNN model is able to learn a convex function that approximates the system dynamics, while the FNN model captures the nonlinearity using a non-convex function. The objective function shown in Fig. \ref{fig:2_a} also indicates that the FNN model exhibits strong non-convexity since several local minimums are located around the global minimum, making it hard to reach the global optimum without being trapped in those local minimums. For the ICNN model that is trained without the non-negative output constraint, the objective function shown in Fig. \ref{fig:2_b} also fails to be convex due to the presence of inward-bending dents. However, the ICNN model trained under the non-negative output constraint leads to a convex objective function, as shown in Fig. \ref{fig:2_c}, which allows efficient implementation of convex optimization algorithms in finding the global optimum.
\begin{rmk}
    It is important to note that while ICNN models offer benefits such as global optimality and stability, they may lose accuracy when applied to highly non-convex functions due to their inherent convexity. However, for processes that are nonlinear yet not highly non-convex, previous studies have shown that strategies such as McCormick envelope-based convex relaxation can obtain better solutions in less computational time for Refinery Operations Planning Problem~\cite{andrade2016oil-convex}. In this context, ICNN models can also serve the purpose of convex relaxation, providing a computationally efficient alternative for ML-based optimization problems while preserving accuracy. Furthermore, it is shown in Ref. \citenum{amos2017inputconvex} that the partially input convex architecture of ICNNs (PICNN) can improve the representability of ICNN models  by ensuring that the output remains a convex function with respect to certain input elements. However, it is acknowledged that ICNNs may lose accuracy for real-world systems with highly nonlinear and non-convex input-output relationships. Therefore, developing ICNN models requires a delicate balance between convexity and representation power to ensure optimal performance for various applications.
    %In practical applications, comparing the testing losses of ICNN models with traditional FNN models can be an effective way to evaluate the performance of ICNN models. If they yield similar accuracy, ICNN models can be considered a good approximation for nonlinear systems. Additionally, partially input convex architecture of ICNN (PICNN) can be utilized to further restore the representation ability of ICNN models by making the output a convex function to some elements of the input\cite{amos2017inputconvex}. Therefore, developing ICNN models requires a delicate balance between convexity and representation power to ensure optimal performance for various applications.
\end{rmk}

\subsection{Explicit ICNN-MPC}\label{Explicit ICNN-MPC}
In our previous work, an explicit ML-MPC framework for a general class of ML models was proposed\cite{wang2024explicitML-MPC}. Specifically, given a nonlinear ML model, a self-adaptive approximation algorithm is first applied to obtain a number of piecewise linear affine functions that approximate the nonlinear behavior of ML model with a sufficiently small error. This approximation approach leads to discretization of the state-space. As the nonlinearity of the ML model is explicitly expressed by linear affine functions at each segment, the original mpNLP problem in the design of explicit ML-MPC is converted to a number of mpQP problems that can be reliably solved.

The solution to each mpQP problem indicates the best control action within the scope of that mpQP problem, as the state-space has been discretized. To develop explicit ICNN-MPC, one can follow the idea of explicit ML-MPC and the design of ICNN-MPC given by Eq. \ref{eq:ICNN-MPC}. Specifically, the design of explicit ICNN-MPC for the $i$th region of the discretized state-space (i.e., $\Omega_i$) is given by the following mpQP problem:
\begin{equation}\label{eq:mpqp}
\begin{aligned}
\mathcal{J}_i=\mathop{\min}_{u_{t|t},\ldots,u_{t+N_p-1|t}}\;&\sum\limits_{k=1}^{N_p}\bar{x}^\top_{t+k|t}\boldsymbol{M}\bar{x}_{t+k|t}+\sum\limits_{k=0}^{N_p-1}u^\top_{t+k|t}\boldsymbol{N}u_{t+k|t}\\
\text{s.t. }~~
    &\bar{x}_{t+k|t}=\boldsymbol{\overline{W}_{i,k}}x_{t|t}+\sum\limits_{j=1}^{k}\boldsymbol{\widehat{W}_{i,k,j}}u_{t+j-1|t}+\widetilde{W}_{i,k},\;k=1,2,\ldots,N_p\\
    &x_{t|t}\in\Omega_{i,x,0}\subseteq\mathbb{R}^n\\
    &u_{t+k|t}\in\Omega_{i,u,k}\subseteq\mathbb{R}^m,\;k=0,1,\ldots,N_p-1\\
\end{aligned}
\end{equation}
where $\boldsymbol{\overline{W}_{i,k}}$ and $\widetilde{W}_{i,k}$ are the coefficients for $x_{t|t}$ and constant term, respectively. $\boldsymbol{\widehat{W}_{i,k,j}}$ is the coefficient associated with the $(j-1)$th control action $u_{t+j-1|t}$ for the prediction of $\bar{x}_{t+k|t}$. 
The region $\Omega_i$ is an orthogonal polyhedral space of $n+N_p\times m$ dimensions consisting of $\Omega_{i,x,0},\Omega_{i,u,0},\Omega_{i,u,1},\ldots,\Omega_{i,u,N_p-1}$ given that the prediction horizon is $N_p$. For $x_{t|t}$ and $u_{t|t},u_{t+1|t},\ldots,u_{t+N_p-1|t}$ that fall into the region $\Omega_i$, the corresponding coefficients $\boldsymbol{\overline{W}_{i}}=\left[\boldsymbol{\overline{W}_{i,1}},\boldsymbol{\overline{W}_{i,2}},\ldots,\boldsymbol{\overline{W}_{i,N_p}}\right]$, $\boldsymbol{\widehat{W}_{i}}=\left[ \boldsymbol{\widehat{W}_{i,1,1}},\boldsymbol{\widehat{W}_{i,2,1}}\;\boldsymbol{\widehat{W}_{i,2,2}},\ldots,\boldsymbol{\widehat{W}_{i,N_p,1}}\;\boldsymbol{\widehat{W}_{i,N_p,2}}\;\ldots\;\boldsymbol{\widehat{W}_{i,N_p,N_p}}\right]$, and $\widetilde{W}_{i}=\left[\widetilde{W}_{i,1},\widetilde{W}_{i,2},\ldots,\widetilde{W}_{i,N_p}\right]$ used for the approximation of the nonlinear system are obtained with the desired approximation accuracy, and they are determined through the linearization of the ICNN models.\par

The number of regions can be quantitatively evaluated based on the setting of the discretization process. If the original state-space $\Omega_0$ (i.e., the non-discrete one) associated with a prediction horizon of $N_p$ is scaled to the range of $[a,b]$ for each variable using min-max normalization and the minimum length of the segment is set to $s$, then the maximum number of regions $\hat{N}_{\Omega}$ can be obtained is $(\frac{b-a}{s})^{n+m\times N_p}$. This value increases exponentially as $n$, $m$, and $N_p$ increase. One potential solution to the curse of dimensionality is to reduce the order of states using feature selection or extraction techniques (similar to reduced-order modeling work, as shown in Ref. \citenum{zhao2022machine} and Ref. \citenum{zhao2023feature}) and then develop explicit ICNN-MPC based on reduced-order states; in this case, $\hat{N}_{\Omega}$ will be reduced by adopting a smaller value of $n$.\par

Eq. \ref{eq:mpqp} can be further reformulated as follows:
\begin{equation}\label{eq:mpQP-obj2}
\begin{aligned}
\mathcal{J}_i=\mathop{\min}_{U_{N_p|t}}\;&U^\top_{N_p|t}\boldsymbol{M_{i,1}}U_{N_p|t}+x^\top_{t|t}\boldsymbol{M_{i,2}}x_{t|t}+x^\top_{t|t}\boldsymbol{M_{i,3}}U_{N_p|t}+M_{i,4}U_{N_p|t}+M_{i,5}x_{t|t}+M_{i,6}\\
\text{s.t. }~~
    &\boldsymbol{A_{i}}x_{t|t}\leq b_i\\
    &\boldsymbol{C_{i}}U_{N_p|t}\leq d_i+\boldsymbol{F_{i}}x_{t|t}
\end{aligned}
\end{equation}
where $U_{N_p|t}=[u_{t|t},u_{t+1|t},\ldots,u_{t+N_p-1|t}]^\top$ is the stacked vector for manipulated input and
\begin{equation*}
\begin{aligned}
\boldsymbol{M_{i,1}}&=
\begin{bmatrix}
\sum\limits_{k=1}^{N_p}\boldsymbol{\widehat{W}^\top_{i,k,1}}\boldsymbol{M}\boldsymbol{\widehat{W}_{i,k,1}}+\boldsymbol{N} & \sum\limits_{k=2}^{N_p}\boldsymbol{\widehat{W}^\top_{i,k,1}}\boldsymbol{M}\boldsymbol{\widehat{W}_{i,k,2}} & \cdots & \boldsymbol{\widehat{W}^\top_{i,N_p,1}}\boldsymbol{M}\boldsymbol{\widehat{W}_{i,N_p,N_p}}\\
\sum\limits_{k=2}^{N_p}\boldsymbol{\widehat{W}^\top_{i,k,2}}\boldsymbol{M}\boldsymbol{\widehat{W}_{i,k,1}} & \sum\limits_{k=2}^{N_p}\boldsymbol{\widehat{W}^\top_{i,k,2}}\boldsymbol{M}\boldsymbol{\widehat{W}_{i,k,2}}+\boldsymbol{N} & \cdots & \boldsymbol{\widehat{W}^\top_{i,N_p,2}}\boldsymbol{M}\boldsymbol{\widehat{W}_{i,N_p,N_p}}\\
\vdots & \vdots & \ddots & \vdots\\
\boldsymbol{\widehat{W}^\top_{i,N_p,N_p}}\boldsymbol{M}\boldsymbol{\widehat{W}_{i,N_p,1}} & \boldsymbol{\widehat{W}^\top_{i,N_p,N_p}}\boldsymbol{M}\boldsymbol{\widehat{W}_{i,N_p,2}} & \cdots & \boldsymbol{\widehat{W}^\top_{i,N_p,N_p}}\boldsymbol{M}\boldsymbol{\widehat{W}_{i,N_p,N_p}}+\boldsymbol{N}\\
\end{bmatrix}\\
\boldsymbol{M_{i,2}}&=\sum\limits_{k=1}^{N_p}\boldsymbol{\overline{W}^\top_{i,k}}\boldsymbol{M}\boldsymbol{\overline{W}_{i,k}}\\
\boldsymbol{M_{i,3}}&=2\times\begin{bmatrix}
     \sum\limits_{k=1}^{N_p}\boldsymbol{\overline{W}^\top_{i,k}}\boldsymbol{M}\boldsymbol{\widehat{W}_{i,k,1}}&\sum\limits_{k=2}^{N_p}\boldsymbol{\overline{W}^\top_{i,k}}\boldsymbol{M}\boldsymbol{\widehat{W}_{i,k,2}}&\cdots&\boldsymbol{\overline{W}^\top_{i,N_p}}\boldsymbol{M}\boldsymbol{\widehat{W}_{i,N_p,N_p}}\end{bmatrix}\\
M_{i,4}&=2\times\begin{bmatrix}
     \sum\limits_{k=1}^{N_p}\widetilde{W}^\top_{i,k}\boldsymbol{M}\boldsymbol{\widehat{W}_{i,k,1}}&\sum\limits_{k=2}^{N_p}\widetilde{W}^\top_{i,k}\boldsymbol{M}\boldsymbol{\widehat{W}_{i,k,2}}&\cdots&\widetilde{W}^\top_{i,N_p}\boldsymbol{M}\boldsymbol{\widehat{W}_{i,N_p,N_p}}\end{bmatrix}\\
M_{i,5}&=2\times\sum\limits_{k=1}^{N_p}\widetilde{W}^\top_{i,k}\boldsymbol{M}\boldsymbol{\overline{W}_{i,k}}\\
M_{i,6}&=\sum\limits_{k=1}^{N_p}\widetilde{W}^\top_{i,k}\boldsymbol{M}\widetilde{W}_{i,k}
\end{aligned}
\end{equation*}
are the coefficients required for constructing the objective function of a standard mpQP problem. The coefficients $\boldsymbol{A_i}$, $\boldsymbol{C_i}$, $\boldsymbol{F_i}$, $b_i$, and $d_i$ are obtained by performing linear transformations for constraints in Eq. \ref{eq:mpqp}, which is trivial and is not shown here. The solutions to Eq. \ref{eq:mpQP-obj2} can be efficiently obtained using Python Parametric OPtimization Toolbox (PPOPT)\cite{kenefake2022PPOPT}. Under the proposed explicit ICNN-MPC framework, finding the best one from the candidate control actions given by all the mpQP problems is equivalent to solving the following convex MIQP problem. However, it should be noted that under the explicit ML-MPC without using ICNN as the predictive model, solving the above problem does not necessarily lead to a convex optimization problem, and therefore,  global optimizers or exhaustive search have to be applied to find the globally optimal solution.
\begin{equation}\label{eq:MIQP}
\begin{aligned}
\mathcal{J}=\mathop{\min}_{u^\ast_{t|t},\ldots,u^\ast_{t+N_p-1|t}}\;&\sum\limits_{k=1}^{N_p}\bar{x}^\top_{t+k|t}\boldsymbol{M}\bar{x}_{t+k|t}+\sum\limits_{k=0}^{N_p-1}u^{\ast\top}_{t+k|t}\boldsymbol{N}u^{\ast}_{t+k|t}\\
\text{s.t. }~~
        \bar{x}_{t+k|t,i}&=F_{ICNN,k}(x_{t|t},u^{\ast}_{t|t,i},\ldots,u^{\ast}_{t+k-1|t,i}),\;i=1,2,\ldots,N_q,\;k=1,2,\ldots,N_p\\
        \bar{x}_{t+k|t}&=\sum^{N_q}_{i=1}\xi_i\bar{x}_{t+k|t,i},\;k=1,2,\ldots,N_p\\
        u^{\ast}_{t+k|t}&=\sum^{N_q}_{i=1}\xi_{i}u^{\ast}_{t+k|t,i},\;k=0,1,\ldots,N_p-1\\
        \sum^{N_q}_{i=1}\xi_i&=1\;,\xi_i\in\{0,1\}
\end{aligned}
\end{equation}
where $u^{\ast}_{t|t,i},u^{\ast}_{t+1|t,i},\ldots,u^{\ast}_{t+N_p-1|t,i}$ are the optimal control actions given by the $i$th mpQP problem with a prediction horizon of $N_p$. $\bar{x}_{t+k|t,i}$ is the absolute value predicted by the $k$th ICNN model trained for predicting the state at time instant $t+k$ using the control actions given by the $i$th mpQP problem. It should be noted that $F_{ICNN,k}$ in Eq. \ref{eq:MIQP} is the original nonlinear convex one because evaluating the objective function along its original surface rather than the approximated surface given by the piecewise linear affine functions can provide better accuracy. $N_q$ is the total number of mpQP problems obtained for the measured $x_{t|t}$. As the entire state-space (i.e., $x_{t},u_{t}$, $u_{t+1},\ldots,u_{t+N_p-1}$) is discretized, the control actions given by the mpQP problems are optimal only within their own domains. The value of $N_q$ can be evaluated by using the number of total regions obtained (i.e., $\hat{N}_{\Omega}$) divided by the the number of total $x$-level sub-regions obtained, which gives $N_q=(\frac{b-a}{s})^{m\times N_p}$.\par

The sum of the binary variables $\xi_1,\xi_2,\ldots,\xi_{N_q}$ in Eq. \ref{eq:MIQP} is equal to 1 since there is only one best control action that gives the smallest value of the objective function among all candidate control actions due to its convexity. For example, if $\xi_{n_q}$ is the only non-zero term among all binary variables in a solution to Eq. \ref{eq:MIQP}, then the control action given by the $\xi_{n_q}$th mpQP problem is considered the best control action that should be applied to the system. From the perspective of an objective function, each binary variable is associated with a control action that is optimal within each segment of the discretized state-space, which greatly reduces the computational burden of finding the best control action as the number of candidate control actions is narrowed down from infinite (i.e., continuous space) to finite (i.e., discretized space). While the possibility of identifying the optimal solution increases with a finer grid, this comes at the cost of significantly increased computational demands. To solve an MIQP problem, the most commonly used method is to employ an advanced optimizer that utilizes built-in optimization algorithms to obtain the optimal solution. Specifically, we first formulate the optimization problem of Eq. \ref{eq:MIQP} in Pyomo, a Python-based open-source software package\cite{hart2011pyomo}. The optimization problem is then solved via Gurobi, an advanced optimizer that supports a variety of problems including MIQP\cite{gurobi}.

\begin{rmk}
The convexity of the MIQP problem of Eq. \ref{eq:MIQP} guarantees that local optimization algorithms such as greedy algorithms that utilize local optimality, can find the global optimum of a convex optimization problem \cite{devore1996greedy}. Starting with a feasible solution, the greedy algorithm compares the current solution with adjacent feasible solutions to determine the starting point for the next iteration. This approach can guarantee the global optimality of the solution produced in the last iteration step when applied to a convex optimization problem. Additionally, an exhaustive approach that explicitly considers all the possible combinations of the binary variables in Eq. \ref{eq:MIQP} is also feasible but not efficient in obtaining the global optimum\cite{wolsey2014MIP}. Each combination leads to a quadratic programming (QP) problem that is degraded from the original MIQP problem. The number of possible combinations of binary variables in Eq. \ref{eq:MIQP} is exactly equal to $N_q$, which is the number of mpQP problems obtained. The greedy algorithm is problem-oriented, as it requires a specific rule to guide the search process, while the exhaustive approach will become computationally intensive if a large $N_q$ is encountered. Therefore, both have limitations in solving a convex MIQP problem.
\end{rmk}

\begin{rmk}
    In this work, we assume that all states are measurable. One way to apply the proposed explicit ICNN-MPC to systems with state measurement data that are not available is to adopt a state estimator. Specifically, in Ref. \citenum{alhajeri2021RNN-output-feedback}, a neural network (NN)-based state estimator was developed for MPC based on output feedback and ML models. The design of NN-based state estimator can be integrated into the proposed explicit ICNN-MPC framework to obtain optimal control actions for systems with unmeasured states. Additionally, explicit output-feedback-based MPC has also been studied by Ref.~\cite{grancharova2011explicit-output-feedback} where multi-parametric nonlinear programming approach is utilized to explicitly solve output-feedback nonlinear MPC problems. Therefore, to implement explicit ICNN-MPC in an output-feedback fashion, a potential strategy is to adopt ICNN as the state estimator and explore the sufficient conditions under which the optimization problem in MPC is convex.
\end{rmk}

\begin{rmk}
    Model-plant mismatch is a critical issue when using deterministic neural network (NN) models (e.g., ICNN) as the predictive model to predict system dynamics in MPC. When this mismatch occurs in a time-dependent manner in a time-variant system, due to, for example, system wear or aging, incorporating system runtime into the NN training process may be one solution to mitigating the issue. In this case, runtime is treated as an uncertainty parameter within the explicit ICNN-MPC framework. In general, if the mismatch arises from external disturbances in the system, a disturbance estimator can be employed to estimate the disturbance value that will be used in the design of MPC. Additionally, online machine learning can be used to update the machine learning model online using real-time data to mitigate the effect of model-plant mismatch due to time-varying disturbances~\cite{zheng2022online-learning,hu2023online}.
\end{rmk}

\section{Application to Chemical Processes}
In this section, we apply explicit ICNN-MPC to two chemical processes to demonstrate its effectiveness in solving real-time MPC problems. Specifically, we first consider a simple chemical reactor example to demonstrate that explicit ICNN-MPC can achieve the desired closed-loop performance that is similar to that under explicit ML-MPC associated with conventional FNN. Then, we consider a more complex chemical process network consisting of two chemical reactors and a flash drum in Aspen Plus Dynamics to control ethylbenzene production. ML models presented in this section are developed using Pytorch\cite{paszke2019pytorch} and trained using Adam Optimizer\cite{kingma2014adam}.

\subsection{Case study 1: Continuous Stirred Tank Reactor}\label{CSTR-1}
\subsubsection{Problem Formulation}\label{case-1-sub1}
In this section, we consider a homogeneous liquid-phase reaction $A\rightarrow B$ that occurs in a continuous stirred tank reactor (CSTR) under the assumption that the CSTR is well-mixed and non-isothermal while the reaction is second-order irreversible and exothermic. The chemical reaction is governed by the following energy and material balance equations:
\begin{subequations}\label{eq:CSTR-1}
\begin{align}
\frac{dC_A}{dt}=\frac{F}{V}(C_{A0}-C_A)-&k_0e^{\frac{-E_a}{RT}}C^2_A \label{eq:CSTR-mb}\\
\frac{dT}{dt}=\frac{F}{V}(T_{0}-T)+\frac{Q}{\rho_LC_pV}&-\frac{\Delta_rH}{\rho_L C_p}k_0e^{\frac{-E_a}{RT}}C^2_A \label{eq:CSTR-eb}
\end{align}
\end{subequations}
where $C_{A0}$ and $T_0$ represent the concentration of $A$ in the inlet flow and the temperature of the inlet flow, respectively. The concentration of $A$ in the CSTR is denoted by $C_A$ while the temperature of the CSTR is denoted by $T$. The CSTR is equipped with a heating jacket with heat input rate denoted as $Q$. The volumetric flow rate of the input feed and the volume of the CSTR are denoted as $F$ and $V$, respectively. The full description of the process parameters shown in Eq.~\ref{eq:CSTR-1} can be found in Table \ref{table:case1} and Ref. \citenum{wu2021processbook}, and is omitted in this section.\par

The controller is designed to maintain the CSTR operation at an unstable steady-state of which the corresponding parameters are $C_{A0_s}=4\;\rm kmol/m^3$, $C_{As}=1.95\;\rm kmol/m^3$. The manipulated vector is $u=[\Delta C_{A0},\;\Delta Q]^\top$ in which the deviation variable forms of the concentration of $A$ (i.e., $\Delta C_{A0}=C_{A0}-C_{A0_s}$) and the heat input rate (i.e., $\Delta Q=Q-Q_{s}$) are considered. The two control inputs, $\Delta C_{A0}$ and $\Delta Q$, are bounded within $[-3.5\;\rm kmol/m^3,\;3.5\;\rm kmol/m^3]$ and $[-5\times 10^5\;\rm kJ/hr,\;5\times 10^5\;\rm kJ/hr]$, respectively. In this case, the state vector is $x=[x_1,\;x_2]^\top=[C_A-C_{As},\;T-T_s]^\top$ to render the origin of the state-space an equilibrium point of Eq. \ref{eq:CSTR-1}. Additionally, the states of the CSTR are assumed to be measurable in real time for feedback control. The sampling period $\Delta_t$ is set to $10^{-2}$ hr, and the prediction horizon is set to two. While the method can handle longer prediction horizons, this short horizon is chosen for demonstration purposes to facilitate the visualization of the results for explicit ICNN-MPC. Sample-and-hold implementation strategy is adopted when applying control actions.
\begin{table}
\centering
\caption{Process parameters of the CSTR}
\label{table:case1}
\begin{tabular}{lll}
\midrule 
    $E=5\times10^4\;\rm kJ/kmol$     &$C_{A0_s}=4\;\rm kmol/m^3$    &$\Delta_rH=-1.15\times10^4\;\rm kJ/kmol$\\
    $V=1\;\rm m^3$                   &$C_{As}=1.95\;\rm kmol/m^3$   &$C_p=0.231\;\rm kJ/kg\cdot K$\\
    $F=5\;\rm m^3/hr$                &$Q_{s}=0.0\;\rm kJ/hr$        &$\rho_L=1000\;\rm kg/m^3$\\
    $R=8.314\;\rm kJ/kmol\cdot K$    &$T_{s}=402\;\rm K$           &$k_0=8.46\times10^6\;\rm m^3/kmol\cdot hr$\\
    $T_0=300\;\rm K$                 &                             & \\
\bottomrule
\end{tabular}
\end{table}
\subsubsection{Implementation Details}\label{case-1-sub2}
We first conduct extensive open-loop simulations based on Eq. \ref{eq:CSTR-mb} and Eq. \ref{eq:CSTR-eb} to generate the training data for ICNN to capture the nonlinear system dynamics of the CSTR. Specifically, we focus on the system dynamics that features $x_1\in[-1.95\;\mathrm{kmol/m^3},1.95\;\mathrm{kmol/m^3}]$ and $x_2\in[-90\;\mathrm{K},90\;\mathrm{K}]$. The raw data is scaled down to the range of $[-1,1]$ to facilitate the training process and is split into three datasets for training, validation, and testing purposes. Two ICNN models of Eq. \ref{eq:icnn} are trained to predict $\bar{x}_{t+1|t}$ and $\bar{x}_{t+2|t}$ using $(x_{t|t},u_{t|t})$ and $(x_{t|t},u_{t|t},u_{t+1|t})$, respectively. Another two FNN models are also trained, but to predict the exact values of the system states (i.e., $x_{t+1|t}$ and $x_{t+2|t}$) for comparison purposes in developing explicit ML-MPC. The model architectures of the ICNN models and the FNN models are shown in Fig. \ref{fig:3}. The red dashed paths and red solid paths represent the data flow of the ICNN models, whereas the green solid paths denote the data flow of the FNN models. Under each prediction scenario (i.e., prediction horizon of one or two), these two models are designed to have the same number of layers and neurons per layer to ensure a fair comparison of their performances. Mean Square Error (MSE) loss is chosen as the criteria to evaluate model performance for the four ML models. Note that the last activation function of the ICNN models is set to ReLU due to the non-negative output constraint, while that of the FNN models is set to linear activation.\par

The MSE values of the four models for training, validation, and testing datasets under different prediction scenarios are shown in Fig. \ref{fig:4}. As seen from the figure, the ICNN models achieve higher MSE values in all datasets than the FNN models because the representation ability of ICNN is limited by the non-negative constraint imposed on the weights. Overfitting is minor as the differences in MSE values between the training and testing sets are less than $10^{-5}$ for all models. Under the explicit ML-MPC framework, the four ML models are linearized using a self-adaptive approximation algorithm to produce piecewise linear affine functions. The approximation error bound required in the algorithm is set to 15\% while the minimum length constraint for preventing over-discretization is set to 0.125. In this case, the numbers of segments needed for approximating $x_{1,t+1|t}$, $x_{2,t+1|t}$, $x_{1,t+2|t}$, and $x_{2,t+2|t}$ in the ICNN models are 1441, 1276, 82909, and 82216, respectively. The corresponding numbers for the FNN models are 2041, 1441, 99037, and 102376. Fig. \ref{fig:5} shows the discretization of $x_{t|t}$ space (i.e., a two-dimensional space consisting of $x_{1,t|t}$, and $x_{2,t|t}$) in predicting $x_{1,t+1|t}$, $x_{2,t+1|t}$, $x_{1,t+2|t}$, and $x_{2,t+2|t}$ using the ICNN models. Each sub-figure in Fig. \ref{fig:5} presents a partial profile of the discretized state-space with respect to different predicted variables. The weight matrices $\boldsymbol{M}$ and $\boldsymbol{N}$ in Eq. \ref{eq:mpqp} are set to $[500\;0;0\;0.5]$ and $[1\;0;0\;8\times10^{-11}]$ for this case study, as $x$ and $u$ are both of two dimensions and their magnitudes need to be taken into account to balance the value of the objective function. The limit for the computation time of the optimizer is set to the same value of the sampling period $\Delta_t$ (i.e., 0.01 hr) for an effective real-time implementation of MPC.

\begin{figure}[h]
\includegraphics[width=\textwidth]{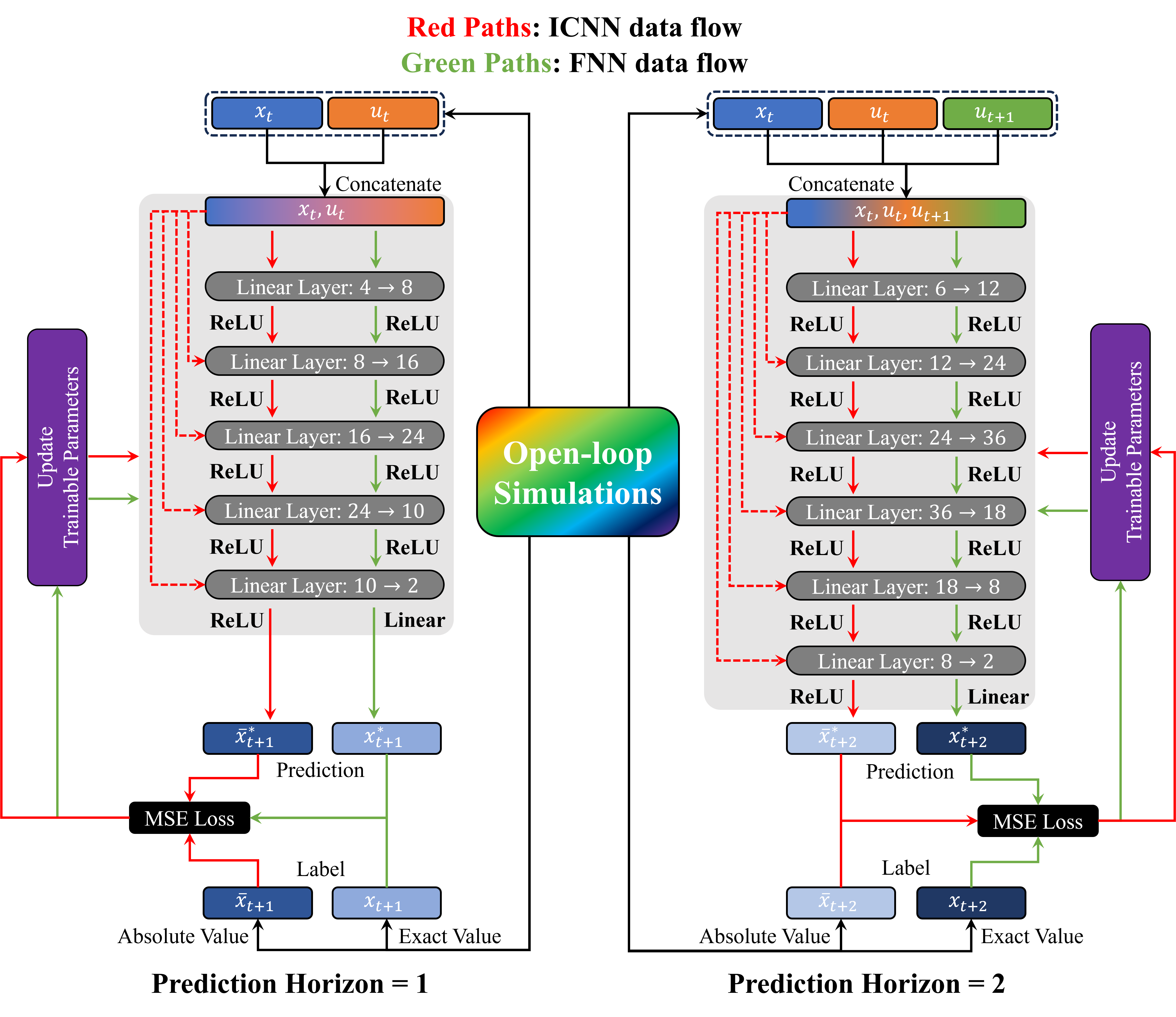}
\centering
\caption{Training strategy and model architectures of ICNN and FNN.}
\label{fig:3}
\end{figure}

\begin{figure}[h]
\centering
\begin{subfigure}{0.495\textwidth}
    \includegraphics[width=\textwidth]{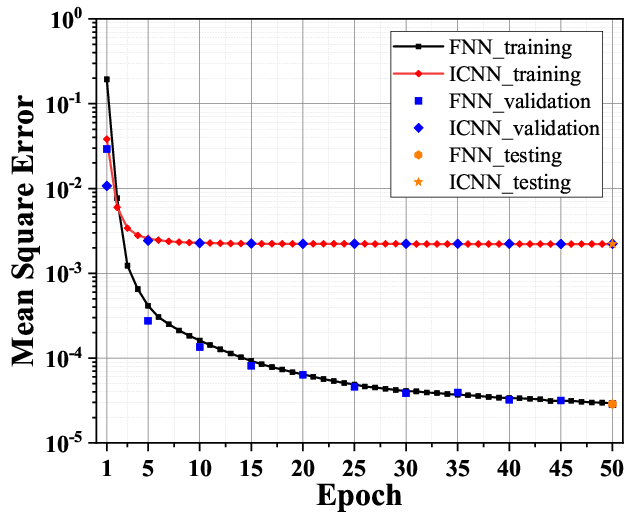}
    \caption{MSE values of ICNN and FNN where prediction horizon equals one.}
    \label{fig:4_a}
\end{subfigure}
\hfill
\begin{subfigure}{0.495\textwidth}
    \includegraphics[width=\textwidth]{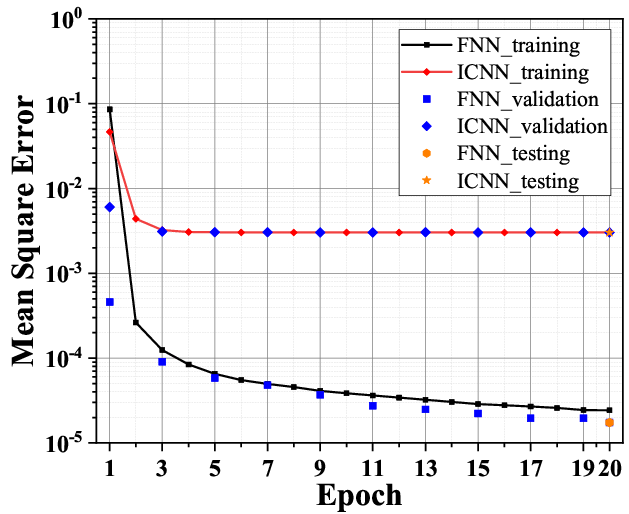}
    \caption{MSE values of ICNN and FNN where prediction horizon equals two.}
    \label{fig:4_b}
\end{subfigure}
\hfill
\caption{MSE values of ICNN and FNN on training, validation, and testing datasets.}
\label{fig:4}
\end{figure}

\begin{figure}[h]
\centering
\begin{subfigure}{0.495\textwidth}
    \includegraphics[width=\textwidth]{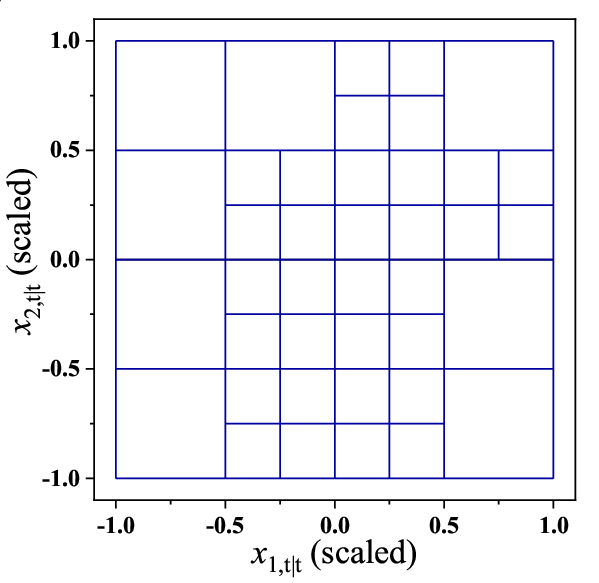}
    \caption{$x$-level discretization of $x_{1,t+1|t}$.}
    \label{fig:5_a}
\end{subfigure}
\hfill
\begin{subfigure}{0.495\textwidth}
    \includegraphics[width=\textwidth]{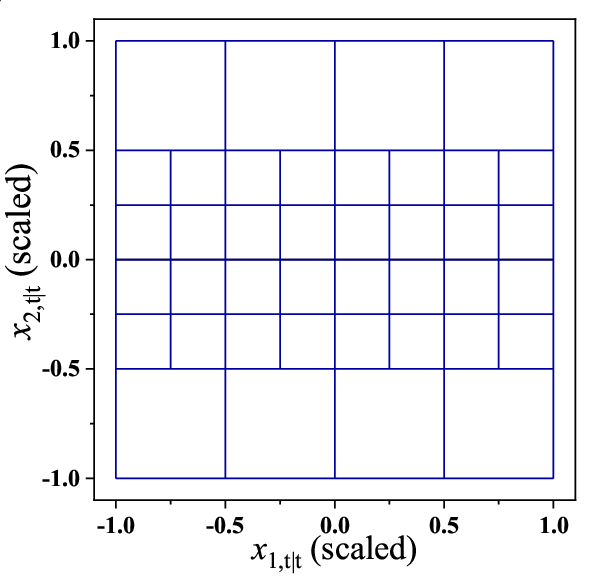}
    \caption{$x$-level discretization of $x_{2,t+1|t}$.}
    \label{fig:5_b}
\end{subfigure}
\hfill
\newline
\begin{subfigure}{0.495\textwidth}
    \includegraphics[width=\textwidth]{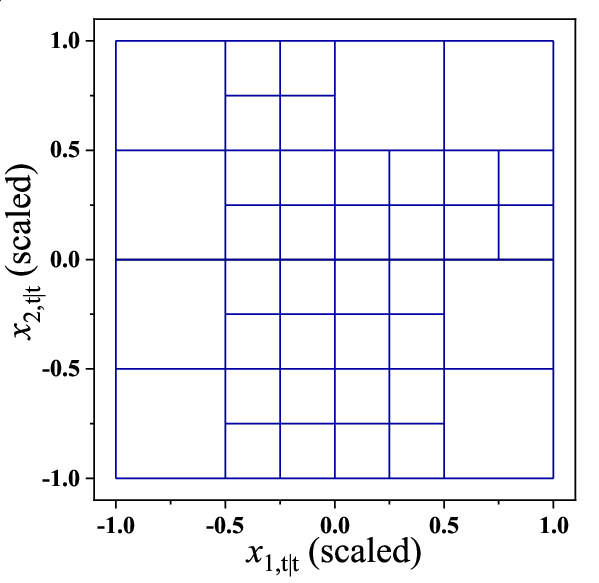}
    \caption{$x$-level discretization of $x_{1,t+2|t}$.}
    \label{fig:5_c}
\end{subfigure}
\hfill
\begin{subfigure}{0.495\textwidth}
    \includegraphics[width=\textwidth]{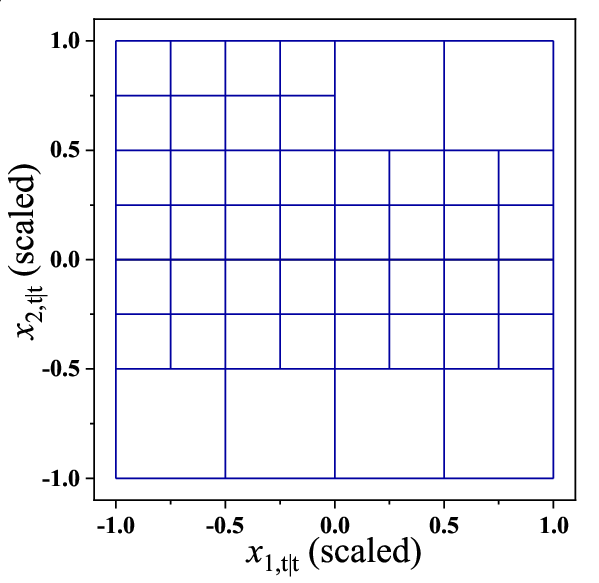}
    \caption{$x$-level discretization of $x_{2,t+2|t}$.}
    \label{fig:5_d}
\end{subfigure}
\caption{Discretization of state-space associated with different piecewise linear affine functions.}
\label{fig:5}
\end{figure}

\begin{figure}[h]
\includegraphics[width=\textwidth]{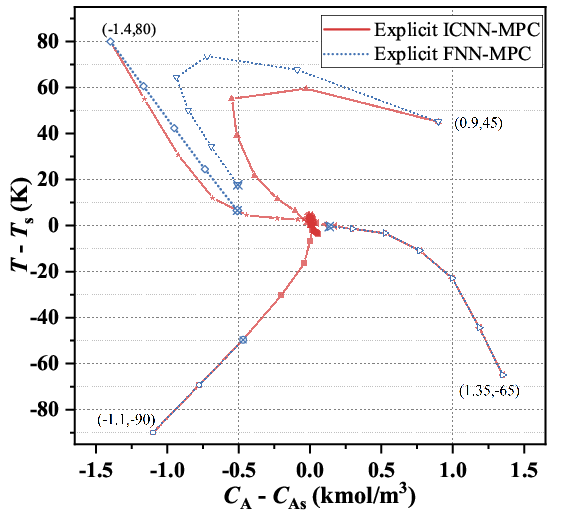}
\centering
\caption{State trajectories starting from four different initial conditions under explicit ICNN-MPC and explicit FNN-MPC, respectively.}
\label{fig:6}
\end{figure}

\begin{figure}[h]
\includegraphics[width=\textwidth]{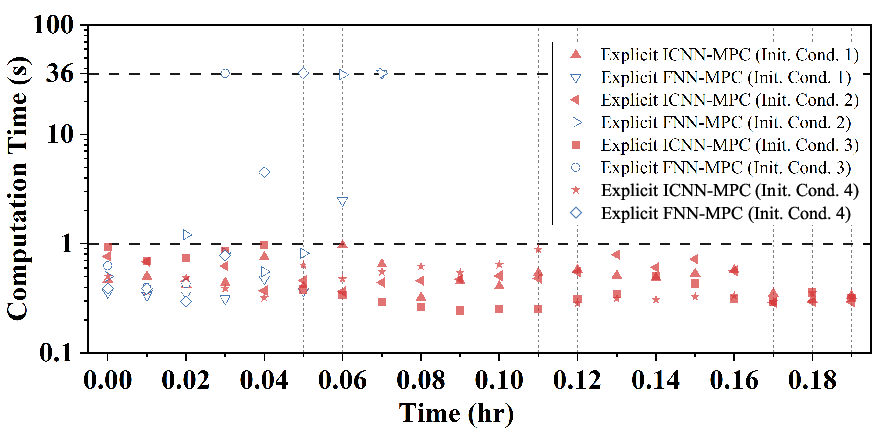}
\centering
\caption{Computation time for solving explicit ML-MPCs at each sampling time under different initial conditions.}
\label{fig:6_s1}
\end{figure}

\begin{figure}[ht]
\centering
\begin{subfigure}{\textwidth}
    \includegraphics[width=\textwidth]{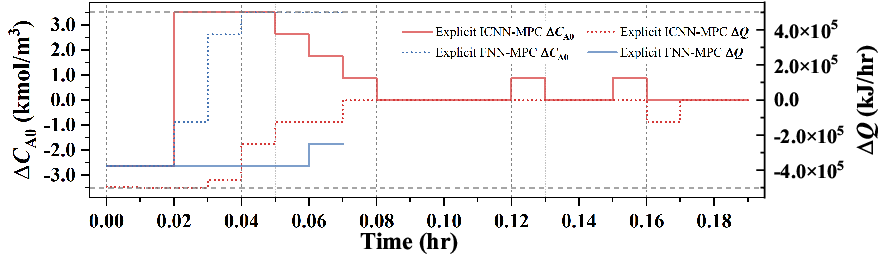}
    \caption{Profiles of control actions for the initial condition $(0.9\;\rm kmol/m^3,45\;\rm K)$ under explicit ICNN-MPC and explicit FNN-MPC.}
    \label{fig:7_a}
\end{subfigure}
\newline
\begin{subfigure}{\textwidth}
    \includegraphics[width=\textwidth]{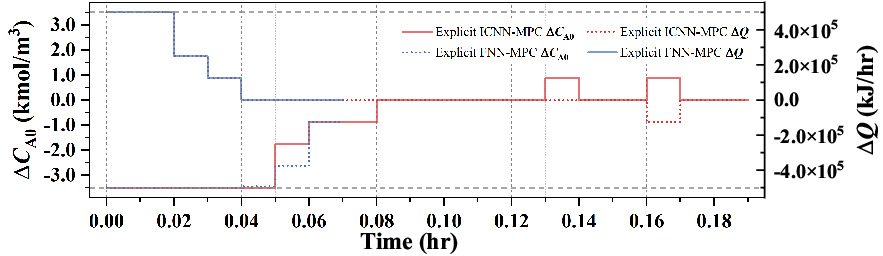}
    \caption{Profiles of control actions for the initial condition $(1.35\;\rm kmol/m^3,-65\;\rm K)$ under explicit ICNN-MPC and explicit FNN-MPC.}
    \label{fig:7_b}
\end{subfigure}
\newline
\begin{subfigure}{\textwidth}
    \includegraphics[width=\textwidth]{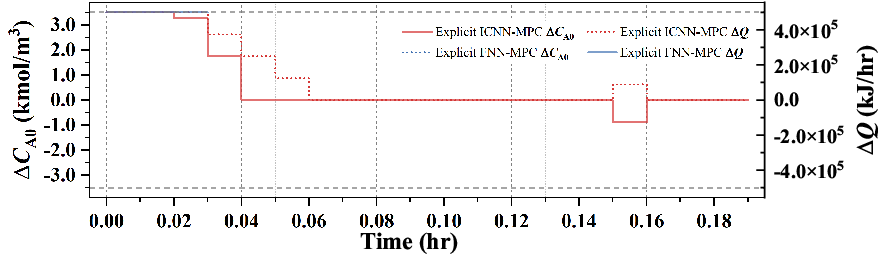}
    \caption{Profiles of control actions for the initial condition $(-1.1\;\rm kmol/m^3,-90\;\rm K)$ under explicit ICNN-MPC and explicit FNN-MPC.}
    \label{fig:7_c}
\end{subfigure}
\newline
\begin{subfigure}{\textwidth}
    \includegraphics[width=\textwidth]{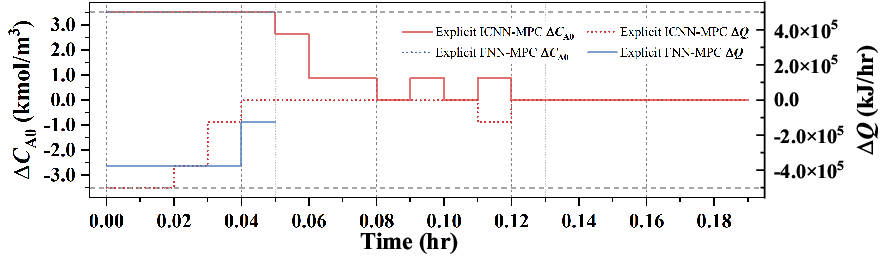}
    \caption{Profiles of control actions for the initial condition $(-1.4\;\rm kmol/m^3,80\;\rm K)$ under explicit ICNN-MPC and explicit FNN-MPC.}
    \label{fig:7_d}
\end{subfigure}
\caption{Comparison of control actions under explicit ICNN-MPC and explicit FNN-MPC.}
\label{fig:7}
\end{figure}

\subsubsection{Simulation Results}\label{case-1-sub3}
Closed-loop simulations for twenty sampling steps are performed using the proposed explicit ICNN-MPC and explicit FNN-MPC. Specifically, the following four different initial conditions: (1) $(0.9\;\mathrm{kmol/m^3}, 45\;\mathrm{K})$, (2) $(1.35\;\mathrm{kmol/m^3}, -65\;\mathrm{K})$, (3) $(-1.1\;\mathrm{kmol/m^3}, -90\;\mathrm{K})$, (4) $(-1.4\;\mathrm{kmol/m^3}, 80\;\mathrm{K})$ are used to demonstrate the superiority of the proposed explicit ICNN-MPC   over the explicit FNN-MPC in finding optimal control actions. As shown in Fig. \ref{fig:6} and Fig. \ref{fig:6_s1}, in each case, the explicit ICNN-MPC successfully drives the system state from all the initial conditions to the origin within twenty sampling steps, while the computation time remains less than one second throughout the entire simulation. However, the time required for solving the explicit FNN-MPC exceeds the time limit (i.e., $10^{-2}$ hr) when the states approach the steady-state, resulting in the termination of subsequent simulations. The system states where the explicit FNN-MPC fail to give the optimal control actions within the pre-set time limit are labeled with a cross sign in Fig. \ref{fig:6}. The corresponding profiles of the control actions given by the two controllers under four different initial conditions are shown in Fig. \ref{fig:7}. Therefore, it is demonstrated that explicit ICNN-MPC is capable of controlling the system states of the CSTR to the desired unstable steady-state while significantly reducing the computation time.

\subsection{Case study 2: Chemical Process Network in Aspen Plus Dynamics}\label{case-2}
\subsubsection{Problem Formulation}\label{case-2-sub1}
In this section, we present a case study in which the proposed explicit ICNN-MPC is applied to a chemical process network simulated in Aspen Plus Dynamics. Specifically, the production of ethylbenzene (EB) via benzene (B) and ethylene (E)   using two  CSTRs and one flash drum is considered, as shwon in Fig. \ref{fig:8}.
\begin{figure}[h]
\includegraphics[width=\textwidth]{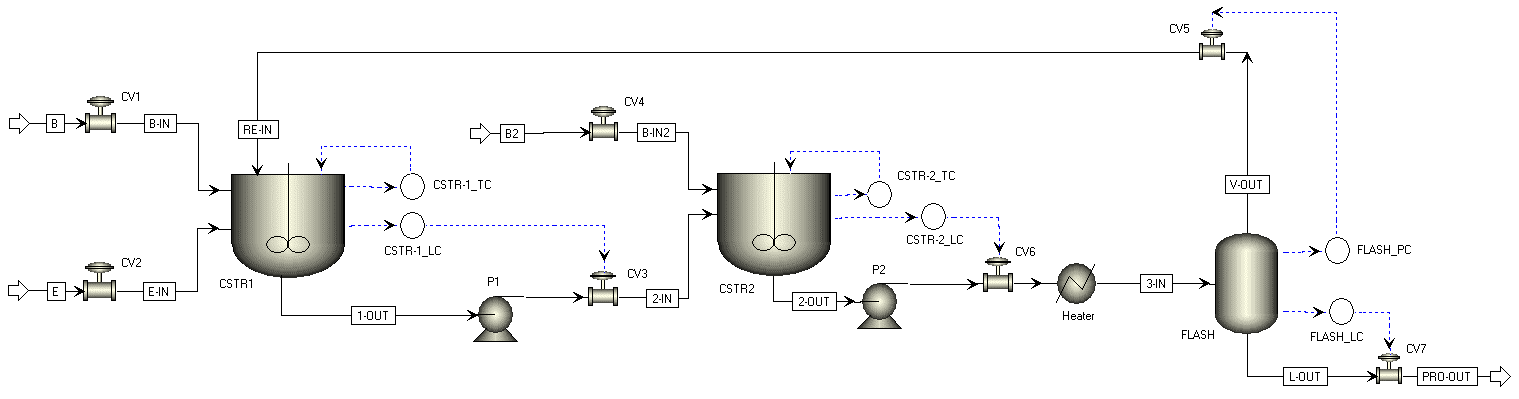}
\centering
\caption{Flow sheet of the ethylbenzene production process modeled by Aspen Plus Dynamics.}
\label{fig:8}
\end{figure}
Diethylbenzene (DEB), a side product of this production process, is generated due to the over substitution of hydrogen atoms in benzene. Additionally, DEB can react with B to produce the desired product EB. The stoichiometric relation and the corresponding Arrhenius equation for the above reactions are taken from Ref. \citenum{alhajeri2022aspendynamics_para} and given as follows:
\begin{subequations}\label{eq:case-2}
\begin{align}
&E+B\rightarrow EB,\;\phantom{2BD}k_1=k_{0,1}e^{\frac{-E_{a,1}}{RT}},\;k_{0,1}=1.528\times10^{6},\;E_{a,1}=71.660\;\mathrm{kJ\cdot mol^{-1}}\label{eq:kinetics-1}\\
&EB+E\rightarrow DEB,\;\phantom{2}k_2=k_{0,2}e^{\frac{-E_{a,2}}{RT}},\;k_{0,2}=2.778\times10^{5},\;E_{a,2}=83.680\;\mathrm{kJ\cdot mol^{-1}} \label{eq:kinetics-2}\\
&DEB+B\rightarrow 2EB,\;k_3=k_{0,3}e^{\frac{-E_{a,3}}{RT}},\;k_{0,3}=4.167\times10^{-1},\;E_{a,3}=62.760\;\mathrm{kJ\cdot mol^{-1}} \label{eq:kinetics-3}
\end{align}
\end{subequations}

\noindent where $k_1$, $k_2$, and $k_3$ are the rate constants for the three reactions, respectively. $k_{0,1}$, $k_{0,2}$, and $k_{0,3}$ are the pre-exponential factors, while $E_{a,1}$, $E_{a,2}$, and $E_{a,3}$ are the molar activation energies. $R$ is the universal gas constant in $\mathrm{J\cdot mol^{-1}\cdot K^{-1}}$ and $T$ is the reaction temperature in K. The temperatures of the two CSTR are fixed at $\mathrm{120\;^\circ C}$ by prompt proportional–integral–derivative (PID) controllers to provide an isothermal condition in case of fluctuation of the inlet flow rate. The flash drum pressure is fixed at 15.9 bar by another PID controller to provide an isobaric condition for the flash evaporation process. The two CSTR and the flash drum are also equipped with level controllers to ensure equipment safety. A heater is used to heat the mixed product flow from the second CSTR (i.e., CSTR2) such that most of the unreacted E can be recycled through an adiabatic flash evaporation process that occurs in the subsequent flash drum. The process parameters for this chemical process are listed in Table \ref{table:case2}.\par
\begin{table}
\centering
\caption{Process parameters of the ethylbenzene production}
\label{table:case2}
\begin{tabular}{llll}
\toprule
    Parameter &Value & Unit &Description\\
\midrule 
    $F_{B}$         &100 &kmol/hr     &Molar flow rate of stream B\\
    $F_{E}$         &150 &kmol/hr     &Molar flow rate of stream E\\
    $F_{B2}$       &60 &kmol/hr      &Molar flow rate of stream B-2\\
    $T_{B}$         &60 &$\mathrm{^\circ C}$   &Temperature of  stream B\\
    $T_{E}$         &60 &$\mathrm{^\circ C}$   &Temperature of  stream E\\
    $T_{B2}$       &60 &$\mathrm{^\circ C}$   &Temperature of stream B2\\
    $P_{B}$         &20 &bar   &Pressure of stream B\\
    $P_{E}$         &20 &bar   &Pressure of  stream E\\
    $P_{B-2}$       &20 &bar   &Pressure of  stream B2\\
    $T_{CSTR1}$    &120 &$\mathrm{^\circ C}$    &Operating temperature of CSTR1\\
    $T_{CSTR2}$    &120 &$\mathrm{^\circ C}$    &Operating temperature of CSTR2\\
    $P_{CSTR1}$    &15 &bar     &Operating pressure of CSTR1\\
    $P_{CSTR2}$    &15 &bar    &Operating pressure of CSTR2\\
    $P_{FD}$        &15.9 &bar   &Pressure of the flash drum\\
    $Q_{0}$         &3.5 &GJ/hr    &Heat duty of the heater under production goal\\
    $X_{EB,0}$      &0.8326 &kmol/koml  &\makecell[ll]{Molar fraction of EB in the stream PRO-OUT \\under production goal}\\
\bottomrule
\end{tabular}
\end{table}

The EB production process is designed to operate at a steady-state where the molar fraction of EB in the stream PRO-OUT (denoted as $X_{EB}$) is 0.8326 kmol/kmol under a constant feeding condition (i.e., the flow rates and temperatures of the inlet flows $B$, $E$, and $B2$ remain unchanged). The heat duty of the heater (denoted as $Q$) is selected to control $X_{EB}$ by changing the temperature of the inlet flow of the flash drum. The deviation variable form for the heat duty of the heater (i.e., $\Delta Q=Q-Q_{0}$) is used as the control input, with a lower bound of 0 GJ/hr and an upper bound of 7 GJ/hr. In this case, $x=X_{EB}-X_{EB,0}$ is the state, and $u=\Delta Q$ is the manipulated input. The origin of the state-space is then rendered as an equilibrium point of the process system shown in Fig. \ref{fig:8}. Additionally, $X_{EB}$ is assumed to be measurable in real time for feedback control. The sampling period $\Delta_t$ is set to 2 hrs, and sample-and-hold implementation strategy is adopted. 

\subsubsection{Implementation Details}\label{case-2-sub2}
We first performed extensive open-loop simulations for the chemical process shown in Fig. \ref{fig:8} via Aspen Plus Dynamics. The generated data is used to train ICNN models to capture the nonlinear system dynamics. The data collected is scaled to the range of [-1,1] by min-max operation and is then used to train two ICNN models that predict the absolute values of the states over one and two sampling periods, respectively. The MSE value of the first ICNN model on the testing set is $8.8\times 10^{-4}$ while that of the second ICNN model is $1.2\times 10^{-3}$. Following the same linearization steps as in our previous work, the state-space of the chemical process is discretized to enable the development of explicit ICNN-MPC in which mpQP problems are solved in each segment to provide the optimal control action for the given system state. The approximation error bound is set to 3\% while the minimum length constraint is set to 0.125. The numbers of segments needed for approximating $x_{t+1|t}$ and $x_{t+2|t}$ in the ICNN models are 79 and 981, respectively. The weight terms $\boldsymbol{M}$ and $\boldsymbol{N}$ in Eq. \ref{eq:mpqp} are set to $5\times 10^{4}$ and $5\times 10^{-2}$ for this case study.\par

In this work, modifying the values and retrieving the data from Aspen Plus Dynamics are realized using a Python library termed ``win32com'', which has been discussed in Ref. \citenum{yamanee2020aspen-python}. The original program in Ref. \citenum{yamanee2020aspen-python} is revised to integrate explicit ICNN-MPC with Apsen dynamic simulation to achieve closed-loop control of chemical processes. Specifically, the new program includes the following four code blocks: \par
\textbf{Code Block 1: Aspen Plus Dynamics Initialization}. The role of this code block is to establish the connection between Aspen Plus Dynamics and Python using the ``win32com'' library. Simulation file with a ``.dynf'' file extension will be opened by the Aspen Plus Dynamics software and data handlers for streams and blocks are obtained in this step for later use.

\textbf{Code Block 2: Data Retrieval and Modification}. The data of streams and blocks can be retrieved and modified in real time by using the obtained handlers. For example, 
the molar flow rate of a substance denoted as ``A'' in the stream ``S1'' can be retrieved and stored in a variable $\mathrm{F_{A}}$ by using a command ``$\mathrm{F_{A}}$\;=\;STREAMS(S1).Tcn(A).Value'' and be changed to a new value $\mathrm{\bar{F}_{A}}$ by using another command ``STREAMS(S1).Tcn(A).Value\;=\;$\mathrm{\bar{F}_{A}}$'' where ``STREAMS'' is the name of the stream handler and ``Tcn'' is the notation for molar flow rate in Aspen Plus Dynamics.

\textbf{Code Block 3: Optimal Control via Explicit ICNN-MPC}. The proposed explicit ICNN-MPC first locates the region where the current state falls into such that the associated candidate control actions can be obtained. A convex MIQP problem is then formulated, and efficiently solved by Gurobi optimizer to give the optimal control action. It should be noted that this code block can be modified solely to implement novel control strategies in Aspen Plus Dynamics without changing the other parts of the program.

\textbf{Code Block 4: Simulation}. This is where information from different code blocks are exchanged and processed in order to apply the control action to the Aspen Plus Dynamics model. After establishing the connection between Aspen Plus Dynamics and Python, the data retrieved from Aspen Plus Dynamics is passed to the explicit ICNN-MPC to obtain the optimal control action and apply it for one sampling period. This ``Retrieval-Optimize-Evolve'' loop continues until the pre-set simulation time limit has been reached.

\begin{figure}[ht]
\includegraphics[height=\textheight]{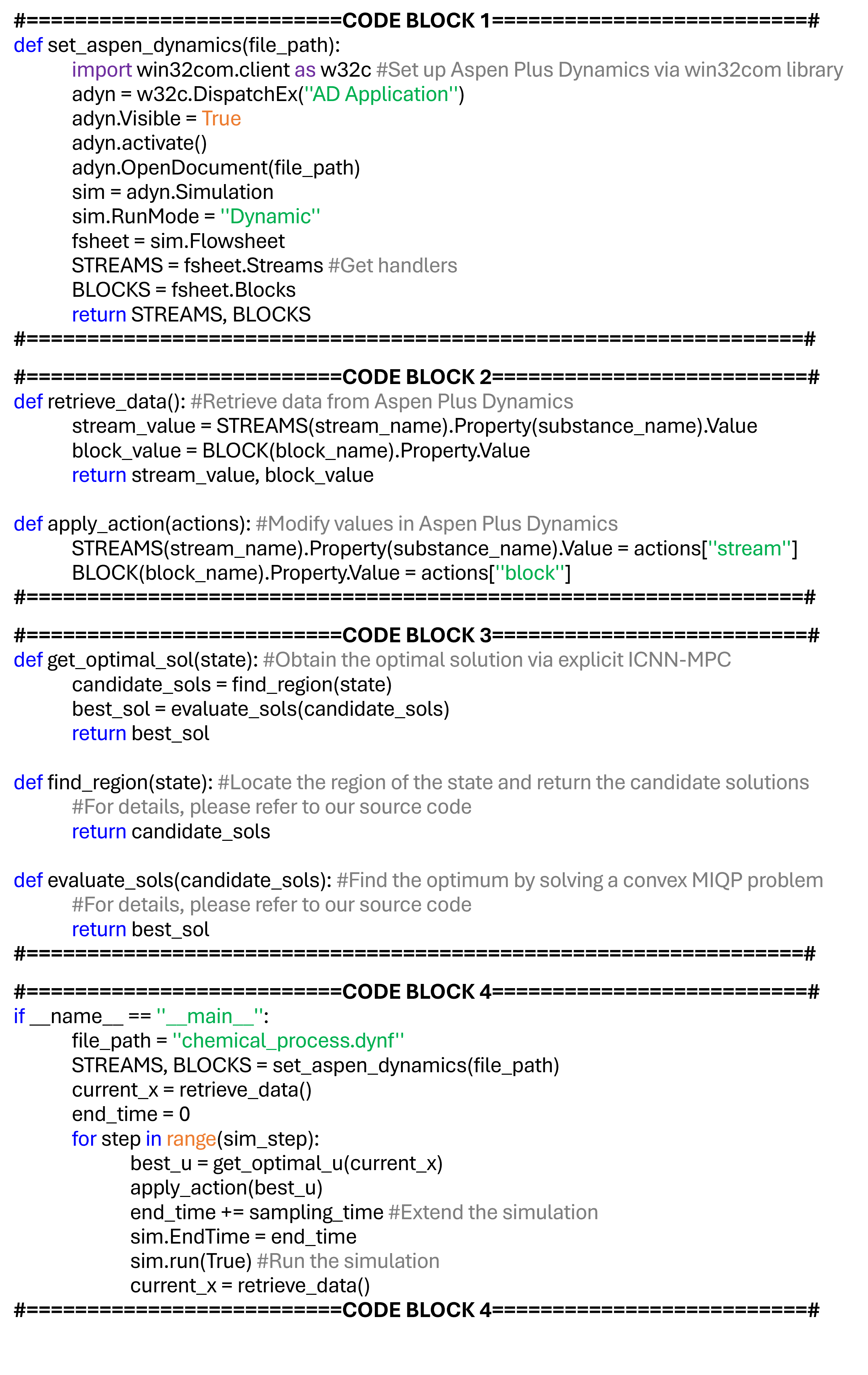}
\centering
\caption{Python script for closed-loop control of chemical processes simulated by Aspen Plus Dynamics via explicit ICNN-MPC.}
\label{fig:9}
\end{figure}

Closed-loop simulations of 30 hrs are conducted in Aspen Plus Dynamics using the explicit ICNN-MPC developed for this case study. Specifically, the simulation is initially carried out under $u=0\;\mathrm{GJ/hr}$ for 3 hrs to stabilize the system at the steady-state. Subsequently, the system is manually deviated from the set-point by changing the heat duty of the heater for a period of 17 hrs. The proposed explicit ICNN-MPC is activated to control the process from the instance of 20 hrs and the goal is to drive the $X_{EB}$ back to the set-point $X_{EB,0}$. Two initial deviations, including $u=-2\;\mathrm{GJ/hr}$ and $u=2.5\;\mathrm{GJ/hr}$, are considered to test the effectiveness of the proposed explicit ICNN-MPC in controlling $X_{EB}$. To better evaluate the performance of the proposed controller, integral absolute error (IAE) is introduced as a criterion due to its ability to take overshoot, oscillation, and transition into account. The equation for calculating IAE is given as follows:
\begin{equation}\label{eq:IAE}
\mathrm{IAE}=\int_{0}^{\mathrm{T}} |e(t)|dt
\end{equation}
where $|e(t)|$ denotes the absolute deviation between the controlled variable $X_{EB}$ and its set-point $X_{EB,0}$ while $T$ is the total time of closed-loop simulation (i.e., 30 hrs for this case study). The controller with a smaller IAE value indicates a better tracking performance. Additionally, settling time $T_s$ is introduced to quantitatively evaluate the convergence of $X_{EB}$ under different controllers. In this work, $T_s$ is defined as the time required for $X_{EB}$ to reach and remain within a $0.3\%$ error band of $X_{EB,0}$.
Controller with a smaller $T_s$ value  is more desirable due to  a faster transient response.
\begin{rmk}
    Although the joint simulation between Aspen Plus Dynamics and Matlab has already been developed for optimization-based control of chemical processes\cite{zhang2024matlab_aspen}, integration of Aspen Plus Dynamics with Python is attracting increasing attention due to the increasing number of Python-based ML models that contribute to the development of ML-MPC\cite{wu2019ML-MPC1,wu2023ML-MPC2,xiao2023ML-MPC3}. The communication between Aspen Plus Dynamics and Python will enable the fast validation of complex ML-based controllers or optimizers with the help of various ML libraries provided in Python community. Additionally, the cross-platform feature of Python allows it to act as a coordinator under which information from different devices can be gathered and processed to achieve a more efficient utilization of resources.
\end{rmk}

\subsubsection{Simulation Results}\label{case-2-sub3}
The trajectories of $X_{EB}$ obtained from closed-loop simulations under different initial conditions are shown in Fig. \ref{fig:10}. For comparison purposes, the closed-loop simulations using open-loop control strategy (i.e., using the steady-state values of control actions for all times) were also conducted for the two different initial conditions. The performance of the proposed explicit ICNN-MPC and open-loop control is evaluated by IAE and $T_s$. The results are listed in Table \ref{table:case2-IAE}. The profiles of control actions are shown in Fig. \ref{fig:11}. As seen in Fig. \ref{fig:10}, the proposed explicit ICNN-MPC is able to drive the $X_{EB}$ to the desired set-point under two different initial conditions with smaller IAE and $T_s$ values than open-loop control. Although open-loop control successfully drives $X_{EB}$ to its set-point, it comes with larger IAE and $T_s$ values, which will introduce more uncertainty into downstream sections over time. Therefore, the effectiveness of explicit ICNN-MPC in regulating $X_{EB}$ to its set-point is demonstrated under different initial conditions.

\begin{table}
\centering
\caption{Summary for closed-loop simulations under different initial conditions}
\label{table:case2-IAE}
\begin{tabular}{lccc}
\toprule
    Control Strategy    &Initial Value / (kmol/kmol) &IAE  &$T_s$ / hrs\\
\midrule 
    Explicit ICNN-MPC   &$X_{EB}=0.793$ &0.0488  &1.90\\
    Open-loop Control   &$X_{EB}=0.793$ &0.0784  &5.60\\
    Explicit ICNN-MPC   &$X_{EB}=0.855$ &0.0352  &4.26\\
    Open-loop Control   &$X_{EB}=0.855$ &0.0449  &7.24\\
\bottomrule
\end{tabular}
\end{table}

\begin{figure}[h]
\includegraphics[width=\textwidth]{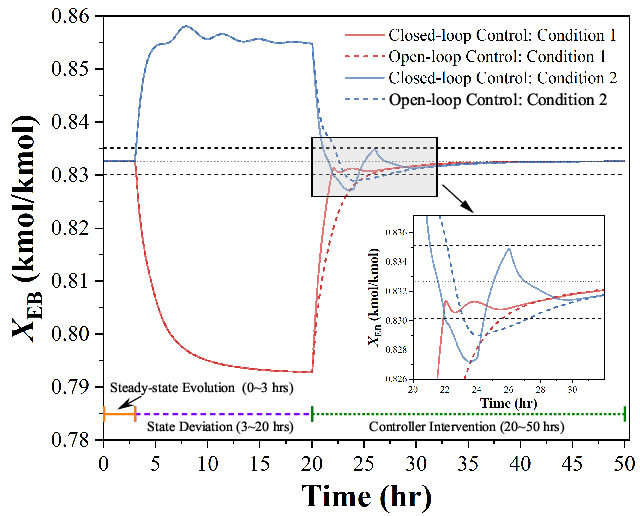}
\centering
\caption{State trajectories obtained from closed-loop simulations under different deviation conditions.}
\label{fig:10}
\end{figure}

\begin{figure}[h]
\centering
\begin{subfigure}{\textwidth}
    \includegraphics[width=\textwidth]{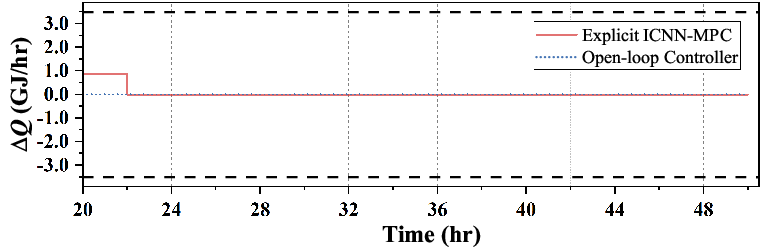}
    \caption{Profiles of control actions under the first deviation condition $\Delta Q=-1.5\;\mathrm{GJ/hr}$.}
    \label{fig:11_a}
\end{subfigure}
\newline
\begin{subfigure}{\textwidth}
    \includegraphics[width=\textwidth]{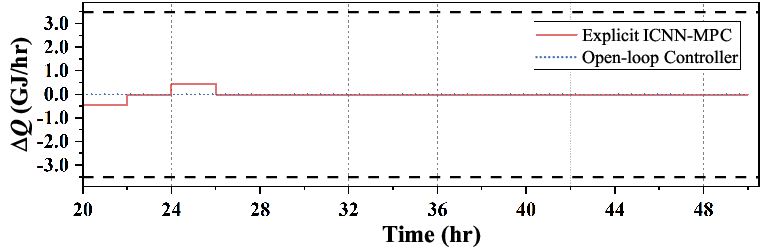}
    \caption{Profiles of control actions under the second deviation condition $\Delta Q=2.5\;\mathrm{GJ/hr}$.}
    \label{fig:11_b}
\end{subfigure}
\hfill
\caption{Profiles of control actions under two different deviation conditions using explicit ICNN-MPC and open-loop controller.}
\label{fig:11}
\end{figure}

\section{Conclusions}
In this work, we utilized ICNN as the predictive model for capturing system dynamics to develop an explicit ICNN-MPC that takes advantage of convex optimization to efficiently determine the optimal control action based on current state measurement. Since the convexity of ICNN does not necessarily guarantee convex optimization problems in explicit ICNN-MPC, we showed the sufficient conditions under which the objective function in explicit ICNN-MPC is guaranteed to be convex, therefore, ensuring global optimal control actions by solving real-time convex MIQP problems. The effectiveness of the proposed explicit ICNN-MPC was demonstrated via two case studies where a single chemical process and a chemical process network simulated by Aspen Plus Dynamics are discussed. The proposed explicit ICNN-MPC was able to regulate the controlled variables while significantly reducing the computation time of solving the optimal control actions.

\section{Acknowledgments}
Financial support from the NRF-CRP Grant 27-2021-0001
 and MOE AcRF Tier 1 FRC Grant (22-5367-
A0001) is gratefully acknowledged.

\providecommand{\latin}[1]{#1}
\makeatletter
\providecommand{\doi}
  {\begingroup\let\do\@makeother\dospecials
  \catcode`\{=1 \catcode`\}=2 \doi@aux}
\providecommand{\doi@aux}[1]{\endgroup\texttt{#1}}
\makeatother
\providecommand*\mcitethebibliography{\thebibliography}
\csname @ifundefined\endcsname{endmcitethebibliography}  {\let\endmcitethebibliography\endthebibliography}{}

\end{document}